\documentclass[11pt]{amsart}    % Arxiv

\usepackage{amssymb}            % Various symbols
\usepackage{graphicx}           % Figures
\usepackage{stmaryrd}           % \boxempty
\usepackage[mathscr]{euscript}  % \mathscr
\usepackage{tikz-cd}            % Commutative diagrams
\usepackage[
    colorlinks=true,
    linkcolor=red!80!black,
    urlcolor=purple,
    citecolor=blue!70!black
  ]{hyperref}

\newcommand{\Z}{\mathbb Z}                % Integers
\newcommand{\R}{\mathbb R}                % Real numbers
\newcommand{\A}{\mathbb A}                % Affine space
\newcommand{\V}{\mathbb V}                % Vector bundle
\newcommand{\N}{\mathrm N}                % Normal bundle
\newcommand{\T}{\mathrm T}                % Tangent bundle
\newcommand{\CC}{\mathscr C}              % Category C
\newcommand{\OO}{\mathscr O}              % Structure sheaf
\newcommand{\PP}{\mathscr P}              % Principal parts

\newcommand{\Spec}{\operatorname{Spec}}   % Spectrum
\newcommand{\Gr}{\mathrm{Gr}}             % Grassmannian
\newcommand{\Sym}{\operatorname{Sym}}     % Symmetric power
\newcommand{\Hom}{\operatorname{Hom}}     % Hom set
\newcommand{\sheafHom}{\mathscr H\! \mathit{om}}  % Hom sheaf
\newcommand{\coker}{\operatorname{coker}} % Cokernel
\newcommand{\im}{\operatorname{im}}       % Image
\newcommand{\pr}{\mathrm{pr}}             % Projection
\newcommand{\id}{\mathrm{id}}             % Identity
\newcommand{\charac}{\operatorname{char}} % Characteristic
\newcommand{\varemptyset}{\varnothing}    % Empty set

\newcommand{\dd}{\mathbf d}               % Intrinsic diff.
\newcommand{\Hess}{\operatorname{Hess}}   % Hessian
\newcommand{\jac}{\operatorname{jac}}     % Jacobian ideal

\newtheorem{theorem}{Theorem}
\newtheorem{proposition}[theorem]{Proposition}
\newtheorem{lemma}[theorem]{Lemma}
\newtheorem{corollary}[theorem]{Corollary}

\theoremstyle{definition}
\newtheorem{definition}[theorem]{Definition}

\theoremstyle{remark}
\newtheorem{remark}[theorem]{Remark}
\newtheorem{example}[theorem]{Example}
\newtheorem{setup}[theorem]{Setup}

\numberwithin{theorem}{section}
\numberwithin{equation}{section}
\begin{document}

\begin{abstract}
The critical loci of a map $f:X\to Y$ between smooth schemes over a field $k$ are the locally closed subschemes $\Sigma^i(f)\subseteq X$ where the differential of $f$ has constant rank.
We prove that if $f : X\to \A^r$ is the general member of a suitably large linear family of maps from a smooth $k$-scheme $X$ to affine space, then the critical loci $\Sigma^i(f)$ are smooth, except in characteristic 2 where the first critical locus $\Sigma^1(f)$ may be singular at a finite set of points.
Moreover, we compute the codimensions of the loci of second order singularities of such general maps $f :X \to \A^r$.
In characteristics different from 2, the codimensions we find agree with those found by Levine in the context of differential topology. 
Finally, assuming that $k$ is an algebraically closed and $\dim X\ge \dim Y$, we give a local description of an arbitrary map $f :X \to Y$ at points of its first critical locus $\Sigma^1(f)$.
In the case of functions and nondegenerate critical points, this description recovers the usual one from Morse theory.
\end{abstract}

% Top matter
\title[Critical loci and second-order singularities]
  {Critical loci and  \\
  second-order singularities \\
  in arbitrary characteristic}
\author{Lucas Braune}
\address{University of Washington, Department of Mathematics, Box 354350, Seattle, WA 98195-4350, USA}
\email{lvhb@uw.edu}
\date{\today}
%\commby{} % Provide editor's name; required by J. Alg. Geom.

\maketitle
\tableofcontents

\section{Introduction}

In this paper we are concerned with the critical loci of maps $f : X\to Y$ between smooth schemes over a field of arbitrary characteristic.
By definition, if $i$ is a nonnegative integer, then the \emph{$i$th critical locus} of such a map is the locally closed determinantal subscheme $\Sigma^i(f)\subseteq X$ where the differential $df : \T_X\to f^* \T_Y$ has rank exactly
$\min(\dim X, \dim Y)-i$.

Our main result asserts the smoothness of the critical loci of the general member of a linear family of maps from a smooth scheme to affine space.
It is an algebraic analogue of a classical result of Thom \cite{Thom1954} according to which the critical loci of a suitably generic map between smooth manifolds are themselves manifolds.

\begin{theorem}
\label{thm:intro-smoothness}
Let $k$ be an infinite field.
Let $X$ be a smooth scheme of finite type and pure dimension $n$ over $k$.
Let 
\begin{equation*}
W\subseteq \Gamma(X, \OO_X^{\oplus r}) = \Hom_k(X,\A^r)
\end{equation*}
be a finite-dimensional linear subspace that separates principal parts of order $2$
(if $k$ is algebraically closed, this means that the natural map $W\to (\OO_X/\mathfrak m_x^{3})^{\oplus r}$ is surjective for all closed points $x\in X$; see Definition \ref{def:sep-pp}).
Let $f\in W$ be a general element and let $i$ be a nonnegative integer.
Then the critical locus $\Sigma^i(f)\subseteq X$ is either empty or of pure codimension $i(|n-r|+i)$ in $X$.
Moreover, $\Sigma^i(f)$ is smooth over $k$, with two exceptions:
\begin{enumerate}
\item The case where $\charac(k)=2$, $i=1$ and $r\ge n$.
\item The case where $\charac(k)=2$, $i=1$, $r=1$ and $n$ is odd.
\end{enumerate}
In both cases, the singular locus of $\Sigma^i(f) = \Sigma^1(f)$ consists of a (possibly empty) finite set of points.
\end{theorem}

Theorem \ref{thm:intro-smoothness} is vacuous when the $k$-scheme $X$ is projective and positive-dimensional, since in this case no linear subspace of $\Gamma(X,\OO_X^{\oplus r})$ can separate principal parts of positive order.
On the other hand, if $X$ is affine, then subspaces of $\Gamma(X,\OO_X^{\oplus r})$ that separate principal parts of any order are guaranteed to exist.

Theorem \ref{thm:intro-smoothness} may be proved in characteristic zero by a simple argument reminiscent of the one used by Thom, see Proposition \ref{prop:gen-crit-dim} below.
It is much more interesting in positive characteristic due to the failure of Sard's lemma.

The only previously known case of Theorem \ref{thm:intro-smoothness} is the case where $r=1$, which is a restatement of Koll\'ar's algebraic Morse lemma \cite[Proposition 18]{Kollar1995}.
To see the connection between Theorem \ref{thm:intro-smoothness} and Morse theory, consider a function $f : X\to \A^1$ and a rational point $x\in X(k)$, and fix \'etale coordinates $x_1,\dotsc, x_n\in \OO_{X,x}$ around $x$.
The ideal in $\OO_{X,x}$ corresponding to the subscheme of critical points $\Sigma^1(f)\subseteq X$ is generated by the partial derivatives
\begin{equation*}
\dfrac{\partial f}{\partial x_1}, \dotsc, 
\dfrac{\partial f}{\partial x_n} \in \OO_{X,x}.
\end{equation*}
Thus, if $x\in \Sigma^1(f)$, then the following are equivalent:
\begin{itemize}
\item Hessian matrix of $f$ is invertible at $x$.
\item The differentials of the partials $\partial f/\partial x_j$ are independent at $x$. 
\item $\Sigma^1(f)$ is nonsingular of codimension $n$ at $x$.
\end{itemize}

In characteristic 2, the Hessian matrix has the peculiarity of being skew-symmetric, hence of even rank.
It can therefore never be invertible when the dimension $n$ of $X$ is odd, which explains exception (2) in Theorem \ref{thm:intro-smoothness}.
The necessity of exception (1) in Theorem \ref{thm:intro-smoothness} is illustrated by the following example:

\begin{example}
Suppose that the base field is algebraically closed of characteristic 2.
Let
\begin{equation*}
W\subseteq  \Gamma(\A^2, \OO_{\A^2}^{\oplus 2})=\Hom(\A^2, \A^2)
\end{equation*}
be the linear space of cubic maps, that is, of maps $\A^2\to \A^2$ whose components are inhomogeneous cubic polynomials in the coordinates of the source $\A^2$. Then $W$ separates principal parts of order $2$. Let $f\in W$ be a general element. By a direct computation,  the critical locus $\Sigma^1(f)\subseteq \A^2$ is a degree-4 curve in $\A^2$ that is singular at exactly one point.
\end{example}

Beyond critical loci, in this paper we consider second-order singularities of maps in the sense of Thom \cite{Thom1954}.
Given a morphism $f : X\to Y$ of smooth schemes over a field $k$ and nonnegative integers $i$ and $j$, we define a locally closed subcheme $\Sigma^{i,j}(f)\subseteq \Sigma^i(f)$ with the property that, if $\Sigma^i(f)$ is smooth over $k$ and of codimension $i(|n-r|+i)$ in $X$, then
\begin{equation*}
\Sigma^{i,j}(f) = \Sigma^j(f|_{\Sigma^i(f)}). 
\end{equation*}
Our definition a scheme-theoretic analogue of the one used by Porteous \cite{Porteous71} in the context of manifolds and nonsingular varieties in characteristic zero.

Our main result about loci of second-order singularities extends Levine's computation \cite[p. 55]{Levine71} of the codimensions of these loci, from suitably generic maps between smooth manifolds, to general members of linear families of maps from a smooth scheme to affine space.

\begin{theorem}
\label{thm:intro-2nd-sings}
Let $k$ be an infinite field. 
Let $X$ be a smooth scheme of finite type and pure dimension $n$ over $k$.
Let 
\begin{equation*}
W\subseteq \Gamma(X, \OO_X^{\oplus r}) = \Hom_k(X,\A^r)
\end{equation*}
be a finite-dimensional linear subspace that separates principal parts of order $2$.
Let $f\in W$ be a general section and let $i$ and $j$ be nonnegative integers. 
Then the locus of second-order singularities $\Sigma^{i,j}(f) \subseteq X$ is either empty or of pure codimension
\begin{equation*}
i(|n-r|+i) +
j(n-m+i-j)(r-m+i-1) +\tfrac 1 2 j(j\pm1)(r-m+i)
\end{equation*}
in $X$, where $m := \min(n, r)$ and the symbol $\pm$ should be read as ``plus'' if $\charac(k)\ne 2$ and as ``minus'' otherwise.
Moreover, if $k$ has characteristic zero, then $\Sigma^{i,j}(f)$ is smooth.
\end{theorem}

If $\charac(k)\ne 2$, then the generic codimensions of Theorem \ref{thm:intro-2nd-sings} agree with those found by Levine in the context of differential topology.
In fact, to prove Theorem \ref{thm:intro-2nd-sings} we are able to use arguments similar to those employed by Golubitsky and Guillemin in their exposition of Levine's result \cite[Chap. VI, \S 3-4]{GG73}.
This is possible because in Theorem \ref{thm:intro-2nd-sings} we do not make smoothness claims outside of characteristic zero. 

The proofs of Theorems \ref{thm:intro-smoothness} and \ref{thm:intro-2nd-sings} are initiated in section \ref{sec:generic-maps} of this paper, where we explain how both theorems follow from results proved later in the paper.
We prove the last of the results upon which Theorem \ref{thm:intro-2nd-sings} depends in section \ref{sec:sym-bil}, and do the same for Theorem \ref{thm:intro-smoothness} in section \ref{sec:min-codim}.

The proofs of both theorems make crucial use of the the second jet scheme $J^2(X,\A^r)$, which we define to be a vector bundle over $X$ whose fiber over a rational point $x\in X(k)$ is the vector space $(\OO_X/\mathfrak m_x^3)^{\oplus r}$.
The sheaf of sections of this vector bundle is Grothendieck's sheaf of principal parts $\PP_X^2(\OO_X^{\oplus r})$. 

In a forthcoming paper we use sheaves of principal parts to construct jet schemes $J^m(X,Y)$ for smooth schemes $Y$ other than affine space.
When carried out in these jet schemes, the arguments of this paper yield generalizations of Theorems \ref{thm:intro-smoothness} and \ref{thm:intro-2nd-sings} to families of maps between any two smooth $k$-schemes.

Although we have not investigated this, it would be interesting to see how the notions of jet schemes and second-order singularities used in this paper compare with the ones introduced by Mount and Villamayor \cite{MV74}.
Mount and Villamayor define jet schemes without reference to sheaves of principal parts, and define singularities using a construction of Boardman \cite{Boardman1967} instead of the one due to Porteous used here.

Another crucial tool in our proof of Theorems \ref{thm:intro-smoothness} and \ref{thm:intro-2nd-sings} is the intrinsic differential of $df : \T_X\to f^* \T_Y$ at $\Sigma^i(f)$. This is a map of locally free $\OO_{\Sigma^i(f)}$-modules
\begin{equation*}
\dd_{\Sigma^i(f)}(df) :
\T_X|_{\Sigma^i(f)} \to 
\sheafHom_{\Sigma^i(f)}(\ker(df|_{\Sigma^i(f)}), \coker(df|_{\Sigma^i(f)}))
\end{equation*}
that generalizes the Hessian bilinear form of a function at at critical point, see section \ref{sec:2nd-order-sings} below.
This map has the property of being surjective if, and only if, the critical locus $\Sigma^i(f)$ is nonsingular and of codimension $i(|n-r|+i)$ in $X$.

Like the Hessian matrix of a function, the second-order differential $\dd_{\Sigma^i(f)}(df)$ exhibits symmetries.
To control its rank in the proofs of Theorems \ref{thm:intro-smoothness} and \ref{thm:intro-2nd-sings}, we compute the dimension of schemes parametrizing linear maps satisfying certain symmetry and rank conditions.
The result of computation is Theorem \ref{thm:univ-deg-bil} below.
For simplicity, here we state a special case that captures the main features of that theorem.

\begin{theorem}
\label{thm:intro-sym-deg}
Let $E$ and $F$ be finite-dimensional vector spaces over a field, and let $A\subseteq E$ be a linear subspace.
Write $e := \dim E$, $f := \dim F$ and $a := \dim A$.
Let $H$ be the vector space of linear maps
\begin{equation*}
h : E\to \Hom(A,F)
\end{equation*}
such that the bilinear map $A\times A\to F$ that sends $(v, w) \mapsto h(v)(w)$ is symmetric.
Let $i$ and $p$ be nonnegative integers.
Let $\Delta^{i,p} \subseteq H$ be the locally closed subscheme parametrizing linear maps $h\in H$ such that
\begin{enumerate}
\item $h$ has rank $\min(e,af)-i$, and
\item $\dim (\ker(h)\cap A) = p$.
\end{enumerate}
If $\Delta^{i,p}$ is nonempty (see Lemma \ref{lemma:nonempty}), then $\Delta^{i,p}$ is smooth of pure codimension
\begin{equation*}
p(n-a+p) +f\cdot 
[ \tfrac 1 2 (-p^2 + p)
+ (e-n)a ] - n(e-n)
\end{equation*}
in $H$, where $n:= \min(e,af)-i$.
\end{theorem}

The last result that we state in this introduction concerns the local structure of a map $f : X\to Y$ between smooth over an algebraically closed field.
At points of the zeroth critical locus $\Sigma^0(f)\subseteq X$, where the differential of $f$ is either injective or surjective, this local structure is completely determined by the inverse function theorem.
Theorem \ref{thm:intro-morse} provides a description of $f$ at points of its first critical locus $\Sigma^1(f)$, assuming that $\dim Y\le \dim X$. 
The description at a given point depends on the unique stratum $\Sigma^{1,j}(f) \subseteq \Sigma^1(f)$ that contains it. 
In the case of functions $f : X\to \A^1$ and nondegenerate critical points, that is, points of $\Sigma^{1,0}(f)$, Theorem \ref{thm:intro-morse} reduces to Morse's Lemma.

\begin{theorem}
\label{thm:intro-morse}
Let $f : X\to Y$ be a morphism of smooth schemes over an algebraically closed field $k$.
Let $x\in \Sigma^1(f)$ be a closed point.
Suppose that $\dim_{f(x)} Y \le \dim_x X$. 

Let $y_1,\dotsc, y_r\in \OO_{Y,f(x)}$ be \'etale coordinates around $f(x)\in Y$, that is, function germs whose differentials form a basis for $\Omega_{Y, y}$ as an $\OO_{Y,f(x)}$-module. 
Let $f_\ell := y_\ell \circ f\in \OO_{X,x}$, where $\ell=1,\dotsc, r$, be the components of $f$ with respect to these coordinates. 
Let $x_1,\dotsc, x_n\in \OO_{X,x}$ be a regular system of parameters such that
\begin{equation*}
(f_1,\dotsc,f_r) = (c_1 + x_1, \dotsc, c_{r-1} + x_{r-1}, f)
\end{equation*}
for suitable constants $c_1,\dotsc, c_{r-1}\in k$;
such a system of parameters is guaranteed to exist after a reordering of $y_1,\dotsc, y_r$. 

Let $j$ be the unique nonnegative integer such that $x\in \Sigma^{1,j}(f)$.
Then $n-r+1-j$ is nonnegative, and is moreover even if $\charac(k)=2$. Furthermore there exists an automorphism of $\widehat \OO_{X,x} = k[[x_1,\dotsc,x_n]]$ as a local $k[[x_1,\dotsc,x_{r-1}]]$-algebra that sends
$
f_r \mapsto  q + h, 
$
where 
\begin{equation*}
q :=
\begin{cases}
x_r^2 + \dotsb + x_{n-j}^2 &
\text{if $\charac(k)\ne 2$} \\
x_r x_{r+1} + \dotsb + x_{n-j-1} x_{n-j} &
\text{if $\charac(k)=2$}
\end{cases}
\end{equation*}
and $h\in k[[x_1,\dotsc,x_{r-1}, x_{n-j+1},\dotsc, x_n]]$ is a power series that does not involve the variables occurring in $q$.
\end{theorem}

The analogue of Theorem \ref{thm:intro-morse} in differential topology is well known. 
It follows from a generalization of Morse's Lemma called ``Morse's Lemma with Parameters''.
This generalization is in turn a consequece of standard results from the theory of finitely deteremined map germs \cite[Theorems 1.2 and 3.4]{Wall81}.

In section \ref{sec:proof-intro-morse} of this paper, we deduce Theorem \ref{thm:intro-morse} from a version of Morse's Lemma with Parameters that holds in positive characteristics, namely Proposition \ref{prop:morse-params} below.
For another version, see \cite[Lemmas 3.9 and 3.12]{GN16}.
We derive Proposition \ref{prop:morse-params} from general statements about power series with finite Milnor number, namely Propositions \ref{DeterminacyBound} and \ref{versal-deform}.
These propositions seem to be folklore, but follow from standard arguments, as we note below.
Their analogues in differential topology are very special cases of \cite[Theorems 1.2 and 3.4]{Wall81}.

\subsection*{Acknowledgements}
Theorems  \ref{thm:intro-2nd-sings} and \ref{thm:intro-morse} first appeared in my Ph.D. thesis \cite{braune-thesis}, where they are used to prove an irrationality result.
It is a pleasure to thank my Ph.D. advisors Eduardo Esteves and S\'andor Kov\'acs for their guidance and support.
Without them this work would not have been possible.
Part of the research described here was conducted during a visit to Leibniz Universit\"at Hannover.
I warmly thank Klaus Hulek and the \emph{Institut f\"ur Algebraische Geometrie} for their hospitality during this visit.
Finally, I would like to thank Daniel Santana for helpful discussions during intial stages of this project.

\section{Vector bundles and degeneracy loci}

In this section we collect definitions and basic results that will be used throughout this paper. 
We begin with our convention for the correspondence between locally free sheaves and vector bundles, which is different from Grothendieck's \cite[Définition II.1.7.8]{EGA}.

Let $X$ be a scheme and let $E$ be a locally free $\OO_X$-module of finite rank.

\begin{definition}
The \emph{vector bundle} associated to $E$ is the $X$-scheme
\begin{equation*}
\V(E) := \Spec_X \Sym(E^\vee).
\end{equation*}
\end{definition}

With this definition, there is a natural isomorphism between $E$ and the sheaf of sections of the projection $\pi : \V(E)\to X$.
In fact, given a morphism of schemes $t : T\to X$, there exists a natural bijection 
\begin{equation}
\label{eqn:vb-sections}
\Hom_X(T, \V(E)) \xrightarrow\sim \Gamma(T, t^* E)
\end{equation}
by the universal mapping properties of the relative spectrum and the symmetric algebra.

\begin{definition}
\label{def:taut-sect}
The \emph{tautological section} $\tau\in \Gamma(\V(E),E_{\V(E)})$ is the section corresponding to the identity morphism of $\V(E)$ under (\ref{eqn:vb-sections}).
\end{definition}

The natural bijection (\ref{eqn:vb-sections}) coincides with the pullback map $f\mapsto f^*\tau$.

\begin{example}
\label{vb-trivial}
Suppose that the $\OO_X$-module $E$ is free with basis $\{v_1,\dotsc,v_e\}\subseteq \Gamma(X,E)$.
Let $\A^e$ be the affine space over $\Spec \Z$ with coordinates $t_1,\dotsc, t_e$.
Then there exists a unique isomorphism of schemes $\V(E) \cong X\times \A^e$ over $X$ with respect to which
\begin{equation*}
\tau = t_1 \cdot \pi^* v_1 + \dotsb + t_e \cdot \pi^* v_e.
\end{equation*}
\end{example}

\begin{remark}
\label{vb-differentials}
The map $E^\vee \to \pi_* \Omega_{\V(E)/X}$ that sends $\sigma \mapsto d(\sigma \cdot \tau)$ is linear over $\OO_X$.
Its adjoint is an $\OO_{\V(E)}$-linear isomorphism
\begin{equation*}
\pi^* E^\vee \xrightarrow\sim \Omega_{\V(E)/X}
\end{equation*}
by Example \ref{vb-trivial} and the computation of the sheaf of differentials on affine space.
\end{remark}

The following simple lemma is at the heart of our main results, Theorems \ref{thm:intro-smoothness} and \ref{thm:intro-2nd-sings}.

\begin{lemma}[Atiyah-Serre]
Let $k$ be an infinite field and let $f : X\to \Spec k$ be a morphism of finite type. 
Suppose that $X$ is pure-dimensional.
Let $Z\subseteq \V(E)$ be a locally closed subscheme of pure codimension $c$.
Let $W\subset \Gamma(X,E)$ be a $k$-linear subspace of finite dimension that generates $E$ as an $\OO_X$-module.  
If $s\in W$ is a general section, then $s^{-1}Z$ is either empty or of pure codimension $c$ in $X$.
Moreover, if $k$ has characteristic zero and $Z$ is smooth over $k$, then $s^{-1}Z$ is smooth over $k$.
\label{serre-bertini}
\end{lemma}

\begin{proof}
Let $\alpha : W\otimes_k \OO_X \to E$ be the $\OO_X$-linear map that sends $s\otimes f\mapsto fs$.
Let $\tilde \alpha : X\times_k W\to \V(E)$ be the map of vector bundles over $X$ induced by $\alpha$.
For each section $s\in W$, we have a commutative diagram with Cartesian squares:
\begin{equation*}
\begin{tikzcd}
s^{-1}Z \ar{r} \ar[d] &
\tilde\alpha^{-1} Z \ar[r] \ar[d] &
Z \ar[d] \\
X \ar[r] \ar[d] &
X\times_k W  \ar[r, "\tilde \alpha"]  \ar[d,"\pr_2"] &
\V(E)\\
\Spec k \ar[r, "s"] & W
\end{tikzcd}
\end{equation*}
There is nothing to show if the second projection $\tilde\alpha^{-1}Z\to W$ is not dominant.
Suppose that it is. 
By hypothesis $\alpha$ is surjective, so $\tilde \alpha$ is smooth and surjective.
Thus the inverse image $\tilde\alpha^{-1}Z$ has pure codimension $c$ in $X\times_k W$, and is smooth over $k$ if $Z$ is.
Applying generic flatness or, in characteristic zero, generic smoothness, to the second projection $\tilde\alpha^{-1}Z\to W$, the result follows.
\end{proof}

We now turn to degeneracy loci.
Let $\alpha : E\to F$ be a map of locally free $\OO_X$-modules of finite rank. 
Let $e$ and $f$ respectively denote the ranks of $E$ and $F$.
Let $m=\min(e,f)$. 
Let $i$ be a nonnegative integer.

\begin{definition}
\label{degeneracy-locus-def}
The \emph{$i$th degeneracy locus} of $\alpha$ is defined to be the subscheme $\Sigma^i(\alpha)\subseteq X$ where exterior power
\begin{equation*}
\wedge^{m-i+1}\alpha : \wedge^{m-i+1} E\to \wedge^{m-i+1} F
\end{equation*}
vanishes if $i\le m+1$, and the empty scheme otherwise.
\end{definition}

A point $x\in X$ lies in $\Sigma^i(\alpha)$ if, and only if, the $k(x)$-linear map $\alpha(x)$ has rank at most $m-i$.
By the Laplace expansion of the determinant, we have closed immersions
\begin{equation*}
\varemptyset = \Sigma^{m+1}(\alpha)\subseteq \Sigma^m(\alpha) \subseteq \dotsb \subseteq 
\Sigma^0(\alpha) = X.
\end{equation*}

\begin{remark}
\label{deg-loci-functor}
If $t: T\to X$ be a morphism of schemes, then
\begin{equation*}
t^{-1}\Sigma^i(\alpha) = \Sigma^i(t^* \alpha)
\end{equation*}
as closed subschemes of $T$.
\end{remark}

Suppose that $0\le i \le m$. Let $\Sigma$ denote the locally closed subscheme $\Sigma^i(\alpha)\setminus \Sigma^{i+1}(\alpha)\subseteq X$.

\begin{proposition}
\label{prop:pts-deg-loc}
A morphism of schemes $t : T\to X$ factors through $\Sigma$ if, and only if, the cokernel of $t^*\alpha : t^*E\to t^*F$ is a locally free $\OO_T$-module of rank $f-m+i$.
\end{proposition}

\begin{proof}
Left to the reader.
The key point is to show that the cokernel of the $\OO_\Sigma$-linear map $\alpha|_\Sigma:E_\Sigma \to F_\Sigma$ is a locally free of rank $f-m+i$. This can be done with the help of Lemma \ref{kleiman} below.
\end{proof}

The next corollary describes canonical isomorphisms that we will often use without mention.

\begin{corollary}
\label{coker-deg-loci}
The kernel, image and cokernel of $\alpha|_\Sigma : E_\Sigma\to F_\Sigma$ are locally free $\OO_\Sigma$-modules of respective ranks $e-m+i$, $m-i$ and $f-m+i$.
If $t:T\to \Sigma$ is a map of schemes, then:
\begin{align*}
\ker (t^* (\alpha|_\Sigma)) &= t^* \ker(\alpha|_\Sigma)\\
\im (t^* (\alpha|_\Sigma)) &= t^* \im(\alpha|_\Sigma)\\
\coker (t^*(\alpha|_\Sigma)) &= t^* \coker(\alpha|_\Sigma)
\end{align*}
\end{corollary}

\begin{proof}
This follows from Proposition \ref{prop:pts-deg-loc} and the following familiar fact.
Let $W$ be a scheme and let
\begin{equation}
\label{ses-locally-free}
\begin{tikzcd}
0 \ar{r} &
A \ar{r} &
B \ar{r} &
C \ar{r} &
0
\end{tikzcd}
\end{equation}
be a short exact sequence of $\OO_W$-modules. If $B$ and $C$ are locally free of finite rank, then $A$ is locally free of finite rank, and the sequence (\ref{ses-locally-free}) remains exact after pullback along any map $t:T\to W$.
\end{proof}

\begin{lemma}
\label{kleiman}
Let $R$ be a ring. 
Let $\beta : M\to N$ be a map of $R$-modules.
Let $A, B\subseteq M$ be submodules such that $\beta(A) \subseteq N$ is a free direct summand of finite rank $a$.
Let $q$ be a nonnegative integer.
The following are equivalent:
\begin{enumerate}
\item The map $\wedge^{q+a} (A+B) \to \wedge^{q+a}N$ induced by $\beta$ is zero.
\item The map $\wedge^q B \to \wedge^q (N/\alpha(A))$ induced by $\beta$ is zero.
\end{enumerate}
\end{lemma}

\begin{proof}
Replacing $A$ and $B$ with their images under $\beta$, it suffices to consider the case where $\beta$ is injective, which is \cite[Lemma 2.5]{Kleiman69}.
\end{proof}

In the case where $X$ is the spectrum of a field, the next result asserts the smoothness and describes the normal bundle of the schemes of matrices with fixed rank and dimensions. 

\begin{proposition}
\label{univ-deg-loci}
Let $\pi : H\to X$ be the vector bundle corresponding to the locally free $\OO_X$-modules $\sheafHom_X(E,F)$.
In symbols,
\begin{equation*}
H = \V(\sheafHom_X(E,F)).
\end{equation*}
Let $h : E_H\to F_H$ be the tautological map, see Definition \ref{def:taut-sect}.
Let $Z$ denote the locally closed degeneracy locus $\Sigma^i(h)\setminus \Sigma^{i+1}(h) \subseteq H$.
Then $Z$ is smooth of relative dimension
\begin{equation*}
(e-m+i)(f-m+i)
\end{equation*}
over $X$.
The canonical $\OO_H$-linear isomorphism
$
\T_{H/X}\xrightarrow\sim \sheafHom_X(E,F)_H
$
of Remark \ref{vb-differentials} induces an $\OO_{\Sigma}$-linear isomorphism
\begin{equation*}
\N_{Z} := (\T_{H/X})_{Z}/\T_{Z/X} \xrightarrow\sim
\sheafHom_{Z}(\ker(h|_{Z}), \coker(h|_{Z})).
\end{equation*}
\end{proposition}

\begin{proof}
This result is well known. 
\end{proof}

\section{The intrinsic differential}

In this section we spell out the scheme-theoretic analogue of Porteous' notion of the intrinsic differential of a map between vector bundles over a smooth manifold \cite{Porteous71}.
The intrinsic differential is used in Porteous' definition of second-order singularities, which we adopt in this paper.
It also plays an important role in the proof of Theorem \ref{thm:intro-smoothness} for reasons that stem from Remark \ref{rmk:id-transversality} below.

Let $k$ be a field.
Let $X$ be a scheme over $k$.
Let $E$ and $F$ be locally free $\OO_X$-modules of ranks $e$ and $f$, respectively.
Let $\alpha : E\to F$ be an $\OO_X$-linear map.
Let $x : T\to X$ be a morphism of schemes.
We think of $x$ as a $T$-valued point of $X$.

\begin{proposition}
\label{prop:id-def}
Suppose the $\OO_X$-modules $E$ and $F$ are free.
Choose bases for $E$ and $F$.
Let
\begin{equation*}
\nabla : \sheafHom(E,F)\to \Omega_X\otimes \sheafHom(E,F)
\end{equation*}
be the $k$-linear map given by differentiation of matrix entries with respect to these bases.
Let
\begin{equation*}
\dd_x\alpha\in 
\Gamma(T, x^*\Omega_X\otimes \sheafHom_T(\ker(x^*\alpha), \coker(x^*\alpha)))
\end{equation*}
be the image of $\nabla \alpha$ under the $\OO_T$-linear map
\begin{equation*}
\begin{tikzcd}
x^*\Omega_X\otimes \sheafHom_T(x^*E,x^*F)
\ar[r,"{(\iota,q)}"] &
x^*\Omega_X\otimes \sheafHom_T(\ker(x^*\alpha), \coker(x^*\alpha))
\end{tikzcd}
\end{equation*}
induced by the inclusion $\iota : \ker(x^*\alpha)\hookrightarrow x^* E$ and the projection $q : x^* F\twoheadrightarrow \coker(x^*\alpha)$. 
Then $\dd_x\alpha$ is independent of the bases used to define it.
\end{proposition}

\begin{proof}
We may assume that $E= \OO_X^{\oplus e}$ and $F=\OO_X^{\oplus f}$ and that the chosen bases on these $\OO_X$-modules are the standard ones.
Let $\varphi : E\xrightarrow\sim E$ and $\psi : F\xrightarrow\sim F$ be $\OO_X$-linear automorphisms.
Let $\tilde \alpha : x^*E \to x^*F$ be the composition $\psi \circ \alpha \circ \varphi^{-1}$.
Consider the following diagram of $\OO_T$-linear maps, where $\tilde \iota$ and $\tilde q$ are the natural inclusion and projection, and $\bar\varphi$ and $\bar\psi$ are induced by $\varphi$ and $\psi$. 
\begin{equation*}
\begin{tikzcd}
0 \ar[r] &
\ker(x^*\alpha) \ar[r, hook, "\iota"] \ar[d, "\bar\varphi"] &
x^* E \ar[r, "x^*\alpha"] \ar[d, "\varphi"] &
x^* F \ar[r, two heads, "q"] \ar[d, "\psi"] &
\coker(x^*\alpha) \ar[r] \ar[d, "\bar\psi"] &
0 \\
0 \ar[r] &
\ker(x^*\tilde \alpha) \ar[r, hook, "\tilde \iota"] &
x^* E \ar[r, "x^*\tilde \alpha"] &
x^* F \ar[r, two heads, "\tilde q"] &
\coker(x^*\tilde \alpha) \ar[r] &
0 
\end{tikzcd}
\end{equation*}
To prove the result, it suffices to show that
\begin{equation*}
\tilde q\cdot \nabla\tilde \alpha\cdot \tilde \iota
=
\bar \psi\cdot q \cdot \nabla\alpha \cdot \iota\cdot \bar \varphi^{-1}
\end{equation*}
as elements of $\Gamma(T, x^*\Omega_X\otimes \sheafHom_T(\ker(x^*\tilde \alpha), \coker(x^*\tilde \alpha)))$.

Identifying $\alpha$, $\varphi$ and $\psi$ with matrices with entries in $\Gamma(X,\OO_X)$ and applying the Leibniz rule, we find that 
\begin{equation*}
d\tilde \alpha   = d\psi \cdot \alpha\cdot \varphi^{-1} +
\psi \cdot d\alpha  \cdot \varphi^{-1} +
\psi\cdot \alpha\cdot d(\varphi^{-1}) 
\end{equation*}
as elements of $\Gamma(X,\Omega_X^{\oplus f\times e})$.
Now
\begin{equation*}
\alpha \cdot \varphi^{-1} \cdot \tilde\iota
= \psi^{-1}\cdot \tilde\alpha\cdot \tilde\iota
= 0
\end{equation*}
in $\Hom_T(\ker(x^*\tilde\alpha), x^* E)$, and
\begin{equation*}
\tilde q\cdot \psi\cdot \alpha =
\tilde q\cdot \tilde\alpha\cdot \varphi = 0
\end{equation*}
in $\Hom_T(x^*F, \coker(x^*\tilde\alpha))$.
Thus
\begin{align*}
\tilde q\cdot d\tilde \alpha  \cdot\tilde\iota
&= \tilde q\cdot\psi \cdot d \alpha  \cdot \varphi^{-1}\cdot\tilde\iota \\
&= \bar\psi \cdot q\cdot d \alpha  \cdot \iota\cdot \bar\varphi^{-1},
\end{align*}
which completes the proof.
\end{proof}

\begin{definition}
\label{def:id}
The \emph{intrinsic differential} of $\alpha : E\to F$ at $x$ is the unique section
\begin{equation*}
\dd_x \alpha \in \Gamma(T, x^* \Omega_X\otimes \sheafHom_T(\ker(x^*\alpha),\coker(x^*\alpha)))
\end{equation*}
such that, for each open subset $U \subseteq X$ over which $E$ and $F$ are free, the restriction of $\dd_x \alpha$ to $x^{-1}U\subseteq T$ coincides with the section of Proposition \ref{prop:id-def} applied to $\alpha|_U : E_U\to F_U$ and $x|_{x^{-1}U} : x^{-1}U\to U$. 
When $X$ is smooth over $k$, we will regard the intrinsic differential as an $\OO_T$-linear map 
\begin{equation*}
\dd_x \alpha : x^* \T_X \to \sheafHom_T(\ker(x^*\alpha),\coker(x^*\alpha)).
\end{equation*}
\end{definition}

Suppose that $X$ is smooth over $k$ and that the cokernel of $x^*\alpha : x^* E\to x^*F$ is a locally free $\OO_T$-module of constant rank.
Then the intrinsic differential $\dd_x\alpha$ may be constructed geometrically, as follows.
Let $\pi : V\to X$ be the vector bundle corresponding to the locally free $\OO_X$-modules $\sheafHom_X(E,F)$.
In symbols,
\begin{equation*}
V = \V(\sheafHom_X(E,F)).
\end{equation*}
Let $h : E_V\to F_V$ be the tautological map (Definition \ref{def:taut-sect}).
Let $\tilde \alpha : X\to V$ be the unique section of $\pi : V\to X$ such that $\tilde \alpha^* h = \alpha$.
Let $i$ be the nonnegative integer defined by 
\begin{equation*}
\operatorname{rank}(\coker(x^* \alpha)) = f-\min(e,f)+i.
\end{equation*}
Let $\Sigma$ denote the locally closed degeneracy locus $\Sigma^i(h)\setminus \Sigma^{i+1}(h)\subseteq V$.
Then $\tilde \alpha \circ x : T\to V$ factors through $\Sigma$ by Proposition \ref{prop:pts-deg-loc}.

\begin{proposition}
With the assumptions and notation of the preceding paragraph, the intrinsic differential $\dd_x \alpha$ is equal to the composition of the $\OO_T$-linear maps
\begin{equation*}
\begin{tikzcd}
x^* \T_X \ar[r,"d\tilde\alpha"] &
(\tilde\alpha\circ x)^* \T_V \ar[r,"q"] &
(\tilde\alpha\circ x)^* \N_\Sigma \ar[r,"\theta","\sim"'] &
\sheafHom_T(\ker(x^*\alpha),\coker(x^*\alpha)),
\end{tikzcd}
\end{equation*}
where $d\tilde\alpha$ denotes the differential of $\tilde \alpha : X\to V$; $\N_\Sigma := \T_V|_\Sigma/\T_\Sigma$ denotes the normal sheaf of $\Sigma$ in $V$; $q$ denotes the canonical projection; and $\theta$ denotes the canonical isomorphism of Proposition \ref{univ-deg-loci} and Corollary \ref{coker-deg-loci}. 
\end{proposition}

\begin{proof}
The question being local on $T$, and therefore on $X$, we may assume that the $\OO_X$-modules $E$ and $F$ are free.
This case is straightforward and left to the reader.
\end{proof}

Definition \ref{def:id} may be motivated by the following observation.

\begin{remark}
\label{rmk:id-transversality}
By definition of transversality, the map $\tilde \alpha : X\to V$ is transverse to $\Sigma$ if, and only if, the intrinsic differential
\begin{equation*}
\dd_{\Sigma^i(\alpha)} \alpha : (\T_X)_{\Sigma^i(\alpha)} \to \sheafHom_{\Sigma^i(\alpha)}(\ker(\alpha|_{\Sigma^i(\alpha)}), \coker(\alpha|_{\Sigma^i(\alpha)}))
\end{equation*}
is surjective.
By \cite[Proposition IV.17.13.2]{EGA} these conditions hold if and only if the scheme-theoretic inverse image $\tilde\alpha^{-1}\Sigma$ is smooth over $k$ and of codimension in $X$ equal to the codimension of $\Sigma$ in $V$. The latter codimension is equal to $i(|e-f|+i)$ by Proposition \ref{univ-deg-loci}.
\end{remark}

\section{Second-order singularities}
\label{sec:2nd-order-sings}

In this section we define the second intrinsic differential and the loci of second-order singularities of a map between smooth schemes over a field, following Porteous \cite{Porteous71}.
We also define a ``bad locus'' that we will use in the proof of Theorem \ref{thm:intro-smoothness}.
After justifying the definitions of these loci of singularities, we examine second-order differentials locally, using coordinates.

Let $k$ be a field.
Let $f : X\to Y$ be a morphism of smooth schemes over $k$.
Let $i$ be a nonnegative integer.

\begin{definition}
\label{def:crit-loci}
The \emph{$i$th critical locus} of $f$ is the locally closed subscheme
\begin{equation*}
\Sigma^i(f) := \Sigma^i(df)\setminus \Sigma^{i+1}(df)  \subseteq X,
\end{equation*}
where $\Sigma^j(df)$ denotes the $j$th degeneracy locus of the differential $df : \T_X\to f^* \T_Y$.
\end{definition}

Let $n$ and $r$ denote the (locally constant) dimension functions of $X$ and $Y$, respectively.
A point $x\in X$ is contained in $\Sigma^i(f)$ if, and only if, the $k(x)$-linear map $df(x) : \T_X(x) \to \T_Y(f(x))$ has rank $\min(n,r)-i$. 

\begin{definition}
\label{def:2nd-id}
Let $T$ be a scheme and let $x : T\to X$ a morphism, which we regard as a $T$-valued point of $X$.
The intrinsic differential of $df : \T_X\to f^* \T_Y$ at $x$ is an $\OO_T$-linear map 
\begin{equation*}
\dd_x( df ) : x^* \T_X \to \sheafHom_T(\ker(x^* df),\coker(x^* df)).
\end{equation*}
The restriction
\begin{equation*}
\dd_x^2 f : \ker(x^* df) \to \sheafHom_T(\ker(x^* df),\coker(x^* df))
\end{equation*}
of this $\OO_T$-linear map to $\ker(x^*df)\subseteq x^* \T_X$ is called the \emph{second intrinsic differential} of $f$ at $x$.
When the morphism $x: T\to X$ is understood from the context, we may write $\dd_T^2 f$ instead of $\dd_x^2 f$ and refer to this $\OO_T$-linear map as the second intrinsic differential of $f$ at $T$.
\end{definition}

The kernel and cokernel of the restriction of the differential $df : \T_X\to f^* \T_Y$ to $\Sigma^i(f)$ are locally free $\OO_{\Sigma^i(f)}$-modules, so the second-order differentials $\dd_{\Sigma^i(f)}(df)$ and $\dd^2_{\Sigma^i(f)} f$ are maps of locally free $\OO_{\Sigma^i(f)}$ modules. Hence we can speak of their degeneracy loci.

Let $j$ be a nonnegative integer.

\begin{definition}
\label{def:2nd-order-sings}
The \emph{bad locus} $B^i(f)$ is the closed subscheme of $\Sigma^i(f)$ defined as follows.
If $n\ge i(|n-r|+i)$, then $B^i(f)$ is the first degeneracy locus
\begin{equation*}
B^i(f):=
\Sigma^1( \dd_{\Sigma^i(f)}(df) ) \subseteq \Sigma^i(f)
\end{equation*}
of the intrinsic differential of $df : \T_X \to f^* \T_Y$ at $\Sigma^i(f)$.
Otherwise, $B^i(f) := \Sigma^i(f)$.
The locally closed subscheme
\begin{equation*}
\Sigma^{i,j}(f) := \Sigma^j( \dd_{\Sigma^i(f)}^2 f)\setminus \Sigma^{j+1}( \dd_{\Sigma^i(f)}^2 f)
\subseteq \Sigma^i(f)
\end{equation*}
is called the \emph{locus of second-order singularites with symbol $(i,j)$}.
\end{definition}

Definition \ref{def:2nd-order-sings} is motivated by the following two results.

\begin{proposition}
\label{prop:transverse-crit}
The bad locus $B^i(f)$ is the locus where $\Sigma^i(f)$ is either not smooth or of codimension in $X$ different from $i(|n-r|+i)$. 
\end{proposition}

\begin{proof}
By Remark \ref{rmk:id-transversality}, the critical locus $\Sigma^i(f)$ is smooth and of codimension $i(|n-r|+i)$ at a point $x\in \Sigma^i(f)$ if, and only if, the intrinsic differential
\begin{equation*}
\dd_{\Sigma^i(f)}(df) : \T_X|_{\Sigma^i(f)} \to \sheafHom_\Sigma(\ker(df|_{\Sigma^i(f)}), \coker(df|_{\Sigma^i(f)}))
\end{equation*}
is surjective at $x$.
The target of $\dd_{\Sigma^i(f)}(df)$ is a locally free $\OO_{\Sigma^i(f)}$-module of rank $i(|n-r|+i)$, so this happens if, and only if, $x\not\in B^i(f)$.
\end{proof}

As $\Sigma^i(f)$ is smooth away from $B^i(f)$, it makes sense to talk about the critical loci of the restriction 
\begin{equation*}
f|_{\Sigma^i(f)\setminus B^i(f)} : \Sigma^i(f)\setminus B^i(f) \to Y.
\end{equation*}

\begin{proposition}
\label{prop:tb-vs-porteous}
Away from $B^i(f)$, the locus of second-order singularities $\Sigma^{i,j}(f)$ agrees with the $j$th critical locus of $f$ restricted to $\Sigma^i(f)$. In symbols:
\begin{equation*}
\Sigma^{i,j}(f)\setminus B^i(f)=
\Sigma^j(f|_{\Sigma^i(f)\setminus B^i(f)})
\end{equation*}
\end{proposition}

\begin{proof}
As we won't need Proposition \ref{prop:tb-vs-porteous} in the sequel, we just give the idea of the proof.
Replacing $X$ with $X\setminus B^i(f)$, it suffices to consider the case in which $B^i(f)$ is empty.
Let $\Sigma := \Sigma^i(f)$, let $K:=\ker(df|_\Sigma)$ and let $C := \coker(df|_\Sigma)$. 
We have a diagram of $\OO_\Sigma$-modules:
\begin{equation*}
\begin{tikzcd}[column sep=large]
&
&
K \ar[d,hook] \ar[rd,"\dd_\Sigma^2 s"] \\
0 \ar[r] &
\T_{\Sigma} \ar[r] \ar[rd,"d(f|_\Sigma)"'] &
\T_X|_{\Sigma} \ar[d, "df|_{\Sigma}" ] \ar[r,"\dd_{\Sigma}(df)"] &
\sheafHom_{\Sigma}(K,C) \ar[r] &
0 \\
&
& f^*\T_Y|_\Sigma
\end{tikzcd}
\end{equation*}
Applying Lemma \ref{kleiman} to $K$ and $\T_\Sigma$ viewed as locally free, locally split $\OO_\Sigma$-submodules of $(\T_X)_\Sigma$, the result follows.
\end{proof}

Let $x\in X(k)$ be a rational point.
The next remark shows that the second-order differentials $\dd_x(df)$ and $\dd_x^2 f$ generalize the Hessian matrix of a function at a critical point.

\begin{remark}
\label{rmk:2nd-id-locally}
Let $y_1,\dotsc, y_r\in \OO_{Y,y}$ be \'etale coordinates near $y$, that is, elements whose differentials form a basis for $\Omega_{Y,y}$ as an $\OO_{Y,y}$-module.
Let $f_\ell := f^\#y_\ell\in \OO_{X,x}$, where $\ell=1,\dotsc, r$, be the components of $f : X\to Y$ with respect to these coordinates.
Let $x_1,\dotsc, x_n\in \OO_{X,x}$ be \'etale coordinates near $x$.
Let
\begin{align*}
\Hess (f) : \T_X(x)\times \T_X(x) \to \T_Y(y) 
\end{align*}
be the \emph{Hessian bilinear map} that sends
\begin{equation*}
\left(\dfrac{\partial}{\partial x_a}, \dfrac{\partial}{\partial x_a}\right) \mapsto
\sum_{\ell=1}^r
\dfrac{\partial^2 f_\ell}{\partial x_a\partial x_b}(x) \cdot \dfrac{\partial}{\partial y_\ell}
\end{equation*}
for all $a, b =1, \dotsc, n$. 
Then the following diagram commutes by definition of the intrinsic differential:
\begin{equation*}
  \begin{tikzcd}
\T_X(x)\otimes \T_X(x) \ar[r, "\Hess (f)"] &
\T_Y(x) \ar[d, two heads]\\
\T_X(x) \otimes \ker(df(x)) \ar[u, hook] \ar[r, "\dd_x(df)"] &
\coker(df(x))
\end{tikzcd}
\end{equation*}
\end{remark}

The Hessian matrix of a function is that is symmetric in characteristic different from 2, and skew-symmetric in characteristic 2.
It will be important to us that the second-order differential $\dd_x(df)$ inherits these symmetries.
To formulate this precisely, given a $k$-scheme $S$ and an $\OO_S$-module $M$, we write 
\begin{equation*}
\Box^2 M :=
\begin{cases}
\Sym^2 M &\text{if $\charac(k)\ne 2$}\\
\wedge^2 M &\text{if $\charac(k)=2$}.
\end{cases}
\end{equation*}
Furthermore, given a submodule $A\subseteq M$ we denote by $M \boxempty A$ the image of $M\otimes A$ under the quotient map $M\otimes M \to \Box^2 M$.

\begin{remark}
\label{rmk:2nd-id-sym}
The map $\Hess (f)$ of Remark \ref{rmk:2nd-id-locally} factors through $\Box^2 \T_X(x)$, so the second-order differential $\dd_x(df)$ lies in the image of the natural inclusion
\begin{align*}
\Hom_k(\T_X(x)\boxempty \ker(df(x)),& \coker(df(x))) \hookrightarrow\\
&\Hom_k(\T_X(x), \Hom_k(\ker(df(x)), \coker(df(x)))). 
\end{align*}
\end{remark}

\section{Jet schemes and sheaves of principal parts}
\label{sec:jets-pp}

The notion of a jet of a map between two manifolds is fundamental in differential geometry.
In algebraic geometry, by contrast, one often works with a slightly different concept, namely that of a principal part of a section of a sheaf of modules.
In the case of morphisms from a scheme $X$ over a field to affine space $\A^r$, morphisms which can be identified with sections of $\OO_X^{\oplus r}$, the two notions make sense and they agree.
In this section we use this observation to give a quick definition of schemes $J^m(X, \A^r)$ parametrizing jets of morphisms from $X$ to affine space, assuming that $X$ is smooth and starting from Grothendieck's sheaves of principal parts \cite[Chap. IV, Part 4]{EGA}.

Over the complex numbers and after analytification, the jet schemes we define coincide with the jet spaces introduced by Ehresmann in differential geometry.
For this reason, our jet schemes coincide in characteristic zero with those constructed algebraically by Mount and Villamayor \cite{MV74}.
It would be interesting to know the extent to which this remains true in positive characteristic.

Before defining jet schemes, we briefly review sheaves of principal parts. We also define and make a few remarks about the notion of separation of principal parts that appears as a hypothesis in Theorems \ref{thm:intro-smoothness} and \ref{thm:intro-2nd-sings}. 

Let $X$ be a scheme over a field $k$.
Let $E$ be an $\OO_X$-module.
Let $m$ be a nonnegative integer.

\begin{definition}
If $x\in X(k)$ is a rational point and $s\in \Gamma(U,E)$ is a section defined on a neighborhood of $x$, then the \emph{principal part of order $m$} of $s$ at $x$ is the image of $s$ under the natural map $\Gamma(U,E) \to E/\mathfrak m_x^{m+1}E$. 
\end{definition}

The $m$th \emph{sheaf of principal parts} associated to $E$ is an $\OO_X$-module $\PP_X^m E$ equipped with an $k$-linear sheaf morphism $d_E^m : E\to \PP_X^m E$, the universal $k$-linear differential operator of order $m$ mapping $E$ to another $\OO_X$-module \cite[Tag 09CH]{stacks-project}.
The universal property satisfied $d_E^m$ characterizes $\PP_X^mE$ as an $\OO_X$-module up to unique isomorphisms.
We write $\PP_X^m$ and $d_X^m$ instead of $\PP_X^m (\OO_X)$ and $d_{\OO_X}^m$.

The sheaf $\PP_X^m E$ derives its name from the following fact: for each rational point $x\in X(k)$, the map $d_E^m$ induces a $k$-linear isomorphism
\begin{equation}
\label{eqn:fiber-pp}
E/\mathfrak m_x^{m+1}E \xrightarrow\sim \PP_X^m E\otimes k(x).
\end{equation}

\begin{definition}
Let $T$ be a scheme.
A \emph{family of principal parts of order $m$ of sections of $E$ over $T$} is a pair $(x,s)$, where $x : T\to X$ is a morphism of schemes and $s\in \Gamma(T, x^* \PP_X^m E)$ is a section.
The \emph{pullback} of such a family $(x, s)$ along a morphism of schemes $u : S\to T$ is the family $(x\circ u, u^* s)$. 
\end{definition}
  
If $x\in X(k)$ is a rational point, then the quotient $\OO_X/\mathfrak m_x^{m+1}$ is not merely a $k$-vector space, but a naturally $k$-algebra.
Similarly, the quotient $E/\mathfrak m_x^{m+1} E$ is not merely a $k$-vector space, but naturally a module over $\OO_X/\mathfrak m_x^{m+1}$.
These facts have counterparts for sheaves of principal parts:

\begin{itemize}
\item The $\OO_X$-module $\PP_X^m$ has a natural structure of $\OO_X$-algebra.
The multiplication map $\PP_X^m \times \PP_X^m \to \PP_X^m$ is the unique $\OO_X$-bilinear map such that
\begin{equation*}
d_X^m u \cdot d_X^m v = d_X^m (uv)
\end{equation*}
for all sections $u, v\in \OO_X$ defined over a common open subset of $X$.
The universal differential operator $d_X^m : \OO_X\to \PP_X^m$ is a map of $k$-algebras, but not of $\OO_X$-algebras in general.

\item The $\OO_X$-module $\PP_X^m E$ has a natural structure of $\PP_X^m$-module. The multiplication map $\PP_X^m \times \PP_X^m E\to \PP_X^m E$ is the unique $\OO_X$-bilinear map such that
\begin{equation*}
d_X^m u \cdot d_E^m s = d_E^m (us)
\end{equation*}
for all sections $u\in \OO_X$ and $s\in E$ defined over a common open subset of $X$. Furthermore, there exists a unique isomorphism of $\PP_X^m$-modules 
\begin{equation*}
\PP_X^m \otimes_{d_X^m, \OO_X} E \xrightarrow\sim \PP_X^m E
\end{equation*}
that sends $\alpha\otimes s \mapsto \alpha \cdot d_E^m s$ for all local sections $\alpha \in \PP_X^m$ and $s\in E$ defined over a common open subset of $X$.
\end{itemize}

Sheaves of principal parts have a number other properties that reflect familiar facts about their fibers.
We will need the following two in the sequel:

\begin{itemize}
\item For each integer $q$ satisfying $0\le q\le m$, there exists a unique $\OO_X$-linear \emph{truncation map}
\begin{equation*}
\varepsilon_{m,q} : \PP_X^m E \to \PP_X^q E
\end{equation*}
such that $\varepsilon_{m,q}\circ d_X^m  = d_X^q$.
The map $d_E^0 : E\to \PP_X^0 E$ is an $\OO_X$-linear isomorphism, which we use to identify $\PP_X^0 E$ with $E$.

\item Suppose that $m\ge 1$.
There exists a unique $\OO_X$-linear map $\iota_m : \Sym^m \Omega_X\otimes E \to \PP_X^m E$ such that 
\begin{equation*}
\iota_m(du_1\dotsb du_m \otimes s) = 
(d_X^m u_1-u_1)\dotsb
(d_X^m u_m-u_m)\cdot d_E^m s
\end{equation*}
for all local sections $u_1,\dotsc,u_m\in \OO_X$ and $s\in E$ defined over a common open subset of $X$.
The sequence of $\OO_X$-modules
\begin{equation*}
\begin{tikzcd}
0 \ar[r] & 
\Sym^m \Omega_X\otimes E \ar[r,"\iota_m"] &
\PP_X^m E \ar[r,"\varepsilon_{m,m-1}"] &
\PP_X^{m-1} E \ar[r] &
0
\end{tikzcd}
\end{equation*}
is right-exact in general, and is exact if either $m=1$, or $X$ is smooth over $k$ and $E$ is locally free.
\end{itemize}

The following definition introduces the notion of separation of jets used as a hypothesis in our main results (Theorems \ref{thm:intro-smoothness} and \ref{thm:intro-2nd-sings}).

\begin{definition}
\label{def:sep-pp}
We say that that a $k$-linear subspace $W \subseteq \Gamma(X,E)$ \emph{separates principal parts of order $m$} if the natural $\OO_X$-linear map $d_E^m(W)\otimes_k \OO_X\to \PP_X^m E$ that sends $s\otimes f \mapsto fs$ is surjective.
\end{definition}

Separation of principal parts is a geometric property that can be defined without reference to Grothendieck's sheaves:

\begin{proposition}
\label{prop:classical-sep-jets}
\begin{enumerate}
\item If $k$ is algebraically closed, then $W\subseteq \Gamma(X, E)$ separates principal parts of order $m$ if, and only if, the natural $k$-linear map $W\to E/\mathfrak m_x^{m+1}E$ is surjective for all closed points $x\in X$.

\item Let $k\subseteq K$ be a field extension. Then $W\subseteq \Gamma(X, E)$ separates principal parts of order $m$ if, and only if, $W\otimes_k K \subseteq \Gamma(X\times_k K, E_{X\times_k K})$ separates principal parts of order $m$.
\end{enumerate}
\end{proposition}

\begin{proof}
The first part is immediate from the isomorphisms (\ref{eqn:fiber-pp}).
Let $X_K := X\times_k K$ and let $u : X_K\to X$ be the first projection.
Let 
\begin{equation*}
\beta : u^* \PP_{X/k}^i E
\xrightarrow\sim
\PP_{X_K/K}^i (u^*E)
\end{equation*}
be the natural $\OO_{X_K}$-linear map, which is induced by the $k$-linear differential operator
\begin{equation*}
\begin{tikzcd}
\OO_X \ar[r,"u^\#"] &
u_* \OO_{X_K} \ar[r, "u_* d^m_{u^*E}"] &
u_* \PP_{X_K}^m(u^* E)
\end{tikzcd}
\end{equation*}
via adjunction. 
It is a standard fact that $\beta$ is an isomorphism.
Let
\begin{align*}
\alpha &: d_E^m (W)\otimes_k \OO_X \to \PP_{X/k}^m E\\
\alpha' &: d_{E_K}^m(W\otimes_k K) \otimes_K \OO_{X_K} \to
\PP_{X_K/K}^m (u^*E)
\end{align*}
be the natural maps.
We wish to show that $\alpha$ is surjective if, and only if, $\alpha'$ is surjective.
This holds because $u$ is faithfully flat and $\alpha' = \beta \circ u^*\alpha$.
\end{proof}

Let $\A^r$ denote the affine space over $k$ with coordinates $t_1,\dotsc, t_r$.

\begin{definition}
Let $x\in X(k)$ be a rational point and let $f : U\to \A^r$ be a morphism of schemes defined on an open neighborhood of $x$ in $X$. 
The \emph{$m$-jet} of $f$ at $x$ is the principal part of order $m$ of the tuple of component functions
\begin{equation*}
(f^\#t_1, \dotsc, f^\#t_r)\in \Gamma(U, \OO_X^{\oplus r})
\end{equation*}
at $x$.
In other words, it is the image of this tuple in $(\OO_X/\mathfrak m_{m+1})^{\oplus r}$.
\end{definition}

\begin{definition}
Let $T$ be a scheme.
A \emph{family of $m$-jets of morphisms from $X$ to $\A^r$ over $T$} is a family of principal parts of order $m$ of sections of $\OO_X^{\oplus r}$ over $T$.
In other words, it is a pair $(x, s)$, where $x : T\to X$ is a morphism of schemes and $s\in \Gamma(T, x^* (\PP_X^m)^{\oplus r})$ is a section.
\end{definition}

We now turn to jet schemes.
Suppose that the scheme $X$ is smooth over $k$, so that the sheaf of principal parts $\PP_X^m$ is locally free of finite rank. 

\begin{definition}
\label{def:jet-scheme}
The \emph{scheme of $m$-jets with source in $X$ and target in $\A^r$}, denoted $J^m(X, \A^r)$, is the total space of the vector bundle over $X$ corresponding to the locally free $\OO_X$-module $(\PP_X^m)^{\oplus r}$. In symbols:
\begin{equation*}
J^m(X, \A^r) := \V((\PP_X^m)^{\oplus r})
\end{equation*}
\end{definition}

Let $\tau\in \Gamma(J^m(X,\A^r), (\PP_X^m)^{\oplus r}_{J^m(X,\A^r)})$ be the tautological section, see Definition \ref{def:taut-sect}.
Thus, given a morphism of schemes $x : T\to X$ and a section $s\in \Gamma(T, x^* (\PP_X^m)^{\oplus r})$, there exists a unique morphism of $X$-schemes $j : T\to J^m(X,\A^r)$ such that $j^* \tau = s$.
In this sense, the family of $m$-jets 
\begin{equation*}
(J^m(X, \A^r)\to X, \tau)
\end{equation*}
is \emph{universal}.

Pullback of the tautological section defines a natural bijection between the set of rational points of $J^m(X,\A^r)$ lying over a rational point $x\in X(k)$ and the $k$-vector space $(\OO_X/\mathfrak m_x^{m+1})^{\oplus r}$. 

\begin{definition}
Let $f : U\to \A^r$ be a morphism fo $k$-schemes whose domain is an open subset of $X$.
The \emph{$m$-jet of f} is unique morphism of $X$-schemes
\begin{equation*}
j^m f : U\to J^m(X,\A^r)
\end{equation*}
such that
$(j^m f)^* \tau = (d_X^m f^\# t_1, \dotsc, d_X^m f^\# t_r)$ as elements of $\Gamma(U, (\PP_X^m)^{\oplus r})$.
If $x : T\to U$ is a morphism of schemes, we write $j_x^m f := j^m f\circ x$ and refer to this morphism as the \emph{$m$-jet of $f$ at $x$}.
\end{definition}

\section{Singularities of generic maps}
\label{sec:generic-maps}

We begin this section by constructing universal critical loci inside first-order jet schemes and using these loci to prove the characteristic-zero case of Theorem \ref{thm:intro-smoothness}.
We then initiate the proofs of the general case of this theorem and of Theorem \ref{thm:intro-2nd-sings} by explaning how these results follow from statements that we prove later in the paper. 

Let $k$ be a field.
Let $X$ be a smooth scheme of pure dimension $n$ over $k$.
Let $\A^r$ be the affine space of dimension $r$ over $k$.
Let $J^1(X, \A^r) = \V((\PP_X^1)^{\oplus r})$ be the first jet scheme and let 
\begin{equation*}
\tau\in \Gamma(J^1(X,\A^r), (\PP_X^1)^{\oplus r}_{J^1(X,\A^r)})
\end{equation*}
be the tautological section, see Definition \ref{def:taut-sect}.
Then
\begin{equation*}
(J^1(X, \A^r)\to X, \tau)
\end{equation*}
is a universal family of $1$-jets. 
This family has a natural differential. 
Indeed, the universal derivation $d : \OO_X\to \Omega_X$ is a differential operator of order 1, so there exists a unique $\OO_X$-linear map $\bar d : \PP_X^1 \to \Omega_X$ such that $\bar d\circ d_X^1 = d$.
We identify
\begin{equation*}
\bar d \tau \in \Gamma(J^1(X,\A^r), (\Omega_X^{\oplus r})_{J^1(X,\A^r)})
\end{equation*}
with an $\OO_{J^1(X,\A^r)}$-linear map $(\T_X)_{J^1(X,\A^r)} \to \OO_{J^1(X,\A^r)}^{\oplus r}$.

\begin{definition}
\label{def:univ-crit-loci}
Let $i$ be a nonnegative integer.
The \emph{$i$th universal critical locus} is the locally closed subscheme
\begin{equation*}
\Sigma^i := \Sigma^i(\bar d \tau) \setminus \Sigma^{i+1}(\bar d \tau) \subseteq J^1(X,\A^r),
\end{equation*}
where $\Sigma^j(\bar d\tau)$ denotes the $j$th degeneracy locus of $\bar d\tau$.
\end{definition}

Definition \ref{def:univ-crit-loci} is justified by the following result.

\begin{proposition}
\label{prop:univ-sings-def-1}
Let $U\subseteq X$ be an open subset and let $f : U\to \A^r$ be a morphism of $k$-schemes.
Let $j^1 f : U\to J^1(X,\A^r)$ be the 1-jet of $f$.
Then $(j^1 f)^{-1}\Sigma^i = \Sigma^i(f)$ as subschemes of $U$.
\end{proposition}

\begin{proof}
Identifying $f$ with an element of $\Gamma(U, \OO_X^{\oplus r})$, we have 
\begin{align*}
(j^1f)^* \bar d\tau =
\bar d ( (j^1f)^* \tau ) =
\bar d  d_X^1 f = df.
\end{align*}
Hence
\begin{align*}
(j^1f)^{-1}\Sigma^i
&= (j^1f)^{-1}(\Sigma^i(\bar d\tau)\setminus \Sigma^{i+1}(\bar d\tau)) \\
&= \Sigma^i((j^1f)^*\bar d\tau)\setminus \Sigma^{i+1}((j^1f)^*\bar d\tau) \\
&= \Sigma^i(df)\setminus \Sigma^{i+1}(df)\\
&= \Sigma^i(f).
\qedhere
\end{align*}
\end{proof}

\begin{proposition}
\label{prop:univ-crit-loci}
The universal critical locus $\Sigma^i\subseteq J^1(X,\A^r)$ is nonempty if, and only if, $0\le i\le \min(n,r)$.
In this case, $\Sigma^i$ is smooth over $X$ and of pure relative codimension
\begin{equation*}
i(|n-r|+i)
\end{equation*}
in $J^1(X,\A^r)$ over $X$.
\end{proposition}

\begin{proof}
Write $J := J^1(X,\A^r)$.
The scheme $\Sigma^i$ is the degeneracy locus of a map between locally free $\OO_J$-modules of ranks $n$ and $r$.
Hence it is empty if $i>\min(n,r)$.
Suppose that $i\le \min(n,r)$.

Let $H$ be the vector bundle corresponding to the locally free $\OO_X$-module $\sheafHom_X(\T_X, \OO_X^{\oplus r})$. In symbols:
\begin{equation*}
H := \V(\sheafHom_X(\T_X, \OO_X^{\oplus r}))
\end{equation*}
Let $\tau\in \Gamma(J, (\PP_X^1)^{\oplus r}_{J})$ be the tautological section, and let $h : (\T_X)_H\to \OO_H^{\oplus r}$ be the tautological linear map.
Let $D : J \to H$ be the morphism of vector bundles over $X$ induced by the $\OO_X$-linear map
\begin{equation*}
\bar d : (\PP_X^1)^{\oplus r}\to \Omega_X^{\oplus r}=\sheafHom_X(\T_X, \OO_X^{\oplus r}).
\end{equation*}
In other words, $D$ is the unique morphism of $X$-schemes such that $D^*h = \bar d \tau$.
We note that $D$ is smooth and surjective, since $\bar d : \PP_X^1 \to \Omega_X$ is surjective. 
Moreover, for each nonnegative integer $j$,
\begin{equation*}
D^{-1}\Sigma^j( h )
= \Sigma^j( D^* h ) 
= \Sigma^j( \bar d \tau ) 
= \Sigma^j.
\end{equation*}
Therefore the result follows from Proposition \ref{univ-deg-loci}, according to which the degeneracy locus $\Sigma^i(h)\setminus \Sigma^{i+1}(h)\subseteq H$ is smooth over $X$ and of pure relative codimension $i(|n-e|+i)$ in $H$ over $X$. 
\end{proof}

\begin{proposition}
\label{prop:gen-crit-dim}
Suppose that $k$ is infinite and $X$ is quasi-compact. 
Let
\begin{equation*}
W\subseteq \Gamma(X, \OO_X^{\oplus r}) = \Hom_k(X,\A^r)
\end{equation*}
be a finite-dimensional linear subspace that separates principal parts of order $1$.
Let $f\in W$ is a general element.
Then the critical locus $\Sigma^i(f) \subseteq X$ is either empty or has pure codimension $i(|n-r|+i)$ in $X$.
Furthermore, if $k$ has characteristic zero, then $\Sigma^i(f)$ is smooth.
\end{proposition}

\begin{proof}
This follows from Lemma \ref{serre-bertini} applied to the $k$-linear subspace
\begin{equation*}
d_X^1(W)\subseteq \Gamma(X, (\PP_X^1)^{\oplus r})
\end{equation*}
and the universal critical locus
\begin{equation*}
\Sigma^{i} \subseteq J^1(X,\A^r) = \V((\PP_X^1)^{\oplus r})
\end{equation*}
in view of Propositions \ref{prop:univ-sings-def-1} and \ref{prop:univ-crit-loci}.
\end{proof}

The proofs of our main theorems are based on that of Proposition \ref{prop:gen-crit-dim}.
We now explain how these theorems follow from results that we will prove later. 

\begin{proof}[Proof of Theorem \ref{thm:intro-smoothness}]
By Propositions \ref{prop:transverse-crit} and \ref{prop:gen-crit-dim}, for general $f\in W$, the singular locus of $\Sigma^i(f)$ is the bad locus $B^i(f)$.
Let 
\begin{equation*}
B^i \subseteq J^2(X,\A^r) = \V((\PP_X^2)^{\oplus r}).
\end{equation*}
be the universal bad locus of Definition \ref{def:univ-2nd-sings}.
For any $f\in W$, we have  $(j^m f)^{-1}B^i = B^i(f)$ by Proposition \ref{prop:univ-sings-def-2}.
The codimension of $B^i$ is established by Theorem \ref{thm:univ-bad-locus}.
Applying Lemma \ref{serre-bertini} to the $k$-linear subspace
\begin{equation*}
d_X^2(W)\subseteq \Gamma(X, (\PP_X^2)^{\oplus r})
\end{equation*}
and $B^i$, the result follows.
\end{proof}

\begin{proof}[Proof of Theorem \ref{thm:intro-2nd-sings}]
Let 
\begin{equation*}
\Sigma^{i,j} \subseteq J^2(X,\A^r) = \V((\PP_X^2)^{\oplus r}),
\end{equation*}
be the universal locus of second-order singularities of Definition \ref{def:univ-2nd-sings}.
For any $f\in W$, we have $(j^m f)^{-1} \Sigma^{i,j} = \Sigma^{i,j}(f)$ by Proposition \ref{prop:univ-sings-def-2}.
The smoothness and codimension of $\Sigma^{i,j}$ are established by Theorem \ref{thm:univ-2nd-sings}.
Applying Lemma \ref{serre-bertini} to the $k$-linear subspace
\begin{equation*}
d_X^2(W)\subseteq \Gamma(X, (\PP_X^2)^{\oplus r})
\end{equation*}
and $\Sigma^{i,j}$, the result follows.
\end{proof}

Our next task is to construct and study the universal loci of singularities used in the proofs of Theorems \ref{thm:intro-smoothness} and \ref{thm:intro-2nd-sings} outlined above.
The construction will be similar to the one given for universal critical loci.
First, we will show that the second-order differentials of Definition \ref{def:2nd-id} make sense for families of $2$-jets of maps from a smooth scheme to affine space.
Then we will define the universal loci of singularities as critical loci of these differentials applied to the universal family of $2$-jets over the second jet scheme.

\section{Intrinsic differentials of principal parts}

In this section we extend the notion of intrinsic differential from maps between locally free sheaves, to families of first-order principal parts of such maps.

Let $k$ be a field.
Let $X$ be a scheme over $k$.
Let $E$ and $F$ be locally free $\OO_X$-modules of finite ranks.
Let $x : T\to X$ be a morphism of schemes and 
\begin{equation*}
\alpha\in \Gamma(T,x^* \PP_X^1 \sheafHom_X(E,F))
\end{equation*}
be a section.
Let $\alpha_0 : x^* E\to x^* F$ be the image of $\alpha$ under the $\OO_X$-linear truncation map $\PP_X^1 \sheafHom_X(E,F)\to \sheafHom_X(E,F)$.

Our goal is to show that the family of first-order principal parts $(x, \alpha)$ has a natural intrinsic differential.
This is clear when $x$ is the inclusion of a rational point:

\begin{remark}
Suppose that $T=\Spec k$ and $x : \Spec k\to X$ is a morphism of $k$-schemes, so that $\alpha$ can be naturally identified with an element of
\begin{equation*}
\sheafHom_X(E, F)\otimes \OO_{X,x}/\mathfrak m_x^2.
\end{equation*}
Let $\tilde\alpha \in \sheafHom_X(E, F)_x$ be a lift of $\alpha$.
It is easy to see that the intrinsic differential
\begin{equation*}
\dd_x \tilde\alpha\in \Omega_X(x)\otimes_k \Hom_k(\ker
\alpha_0, \coker\alpha_0)
\end{equation*}
is independent of $\tilde\alpha$ in the sense that the intrinsic differential at $x$ of any other lift is equal to $\dd_x\tilde\alpha$.
Therefore it is natural to define $\dd_x \alpha := \dd_x \tilde \alpha$.
\end{remark}

\begin{proposition}
\label{prop:id-pp-def}
Suppose the $\OO_X$-modules $E$ and $F$ are free.
Choose bases for $E$ and $F$.
Let
\begin{equation*}
\nabla : \sheafHom_X(E,F)\to \Omega_X\otimes \sheafHom_X(E,F)
\end{equation*}
be the $k$-linear map given by differentiation of matrix entries with respect to these bases.
Let
\begin{equation*}
\overline \nabla : \PP_X^1 \sheafHom_X(E,F)\to \Omega_X\otimes \sheafHom_X(E,F)
\end{equation*}
be the unique $\OO_X$-linear map such that $\overline\nabla \circ d^1_{\sheafHom(E,F)} = \nabla$.
Let
\begin{equation*}
\dd_x \alpha
\in \Gamma( T, x^* \Omega_X\otimes\sheafHom_T(\ker\alpha_0,\coker\alpha_0))
\end{equation*}
be the image of $\overline\nabla \alpha$ under the $\OO_T$-linear map
\begin{equation*}
\begin{tikzcd}
x^*\Omega_X\otimes \sheafHom_T(x^*F, x^*F)
\ar[r,"{(\iota,q)}"] &
x^*\Omega_X\otimes
\sheafHom_T(\ker \alpha_0, \coker \alpha_0)
\end{tikzcd}
\end{equation*}
induced by the inclusion $\iota : \ker \alpha_0 \hookrightarrow x^*E$ and the projection $q: x^* F\to \coker\alpha_0$. Then $\dd_x \alpha$ is independent of the bases used to define it.
\end{proposition}

\begin{proof}
Throughout this proof we will make use of the fact that, for each $\OO_X$-module $M$, there is a natural isomorphism of $\PP_X^1$-modules 
\begin{equation*}
\PP_X^1 M = \PP_X^1 \otimes_{d_X^1, \OO_X} M,
\end{equation*}
see section \ref{sec:jets-pp} of this paper.
This isomorphism shows that the construction $M\mapsto \PP_X^1 M$ extends naturally to a functor from the category of $\OO_X$-modules to that of $\PP_X^1$-modules.

To prove the result, we may assume that $E= \OO_X^{\oplus e}$ and $F=\OO_X^{\oplus f}$ and that the chosen bases on these $\OO_X$-modules are the standard ones.
Let $\varphi : E\xrightarrow\sim E$ and $\psi : F\xrightarrow\sim F$ be $\OO_X$-linear automorphisms.
Let $\Phi : \PP_X^1 E \xrightarrow\sim \PP_X^1 E$ and $\Psi : \PP_X^1 F\xrightarrow\sim \PP_X^1 F$ be the $\PP_X^1$-linear automorphisms induced by $\varphi$ and $\psi$.
Viewing $\varphi$ and $\psi$ as invertible matrices with entries in $\Gamma(X, \OO_X)$, and $\Phi$ and $\Phi$ as invertible matricies with entries in $\Gamma(X, \PP_X^1)$, we have $\Phi = d_X^1 \varphi$ and $\Psi = d_X^1 \psi$.

Let $\tilde \alpha$ denote the section
\begin{equation*}
\tilde \alpha := \Psi \circ \alpha \circ \Phi^{-1}\in \Gamma(T,x^* \PP_X^1 \sheafHom_X(E,F)).
\end{equation*}
Let $\tilde\alpha_0 : x^* E\to x^* F$ be the image of $\tilde\alpha$ under the $\OO_X$-linear truncation map $\PP_X^1 \sheafHom_X (E, F)\to \sheafHom_X(E, F)$.
Consider the following diagram of $\OO_T$-linear maps, where $\tilde \iota$ and $\tilde q$ are the natural inclusion and projection, and $\bar\varphi$ and $\bar\psi$ are induced by $\varphi$ and $\psi$. 
\begin{equation*}
\begin{tikzcd}
0 \ar[r] &
\ker(\alpha_0) \ar[r, hook, "\iota"] \ar[d, "\bar\varphi"] &
x^* E \ar[r, "\alpha_0"] \ar[d, "\varphi"] &
x^* F \ar[r, two heads, "q"] \ar[d, "\psi"] &
\coker(\alpha_0) \ar[r] \ar[d, "\bar\psi"] &
0 \\
0 \ar[r] &
\ker(\tilde \alpha_0) \ar[r, hook, "\tilde \iota"] &
x^* E \ar[r, "\tilde \alpha_0"] &
x^* F \ar[r, two heads, "\tilde q"] &
\coker(\tilde \alpha_0) \ar[r] &
0 
\end{tikzcd}
\end{equation*}
It suffices to show that
\begin{equation*}
\tilde q\cdot \overline\nabla\tilde \alpha\cdot \tilde \iota
=
\bar \psi\cdot q \cdot \overline\nabla\alpha \cdot \iota\cdot \bar \varphi^{-1}
\end{equation*}
as elements of $\Gamma(T, x^*\Omega_X\otimes \sheafHom_T(\ker(\tilde \alpha_0), \coker(\tilde \alpha_0)))$.

Let $d : \OO_X\to \Omega_X$ be the universal derivation.
Let $\bar d : \PP_X^1\to \Omega_X$ be the unique $\OO_X$-linear map such that $\bar d \circ d_X^1 = d$.
Identifying $\alpha$ with a matrix with entries in $\Gamma(T, x^* \PP_X^1)$ and applying the Leibniz rule (Proposition \ref{prop:leibniz-pp} below), we find that
\begin{align*}
\bar d\tilde \alpha  
&= \bar d (d_X^1 \psi) \cdot \alpha_0 \cdot \varphi^{-1} +
\psi \cdot \bar d \alpha \cdot \varphi^{-1} +
\psi \cdot \alpha_0 \cdot \bar d (d_X^1\varphi^{-1})
\end{align*}
as elements of $\Gamma(T,x^*\Omega_X^{\oplus f\times e})$.
Now 
\begin{equation*}
\alpha_0 \cdot\varphi^{-1}\cdot \tilde\iota =
\psi^{-1}\cdot \tilde\alpha_0\cdot \tilde\iota = 0
\end{equation*}
in $\Hom_T(\ker(\tilde\alpha_0), x^* E)$, and
\begin{equation*}
\tilde q\cdot \psi\cdot \alpha_0 =
\tilde q\cdot \tilde\alpha_0 \cdot \varphi = 0
\end{equation*}
in $\Hom_T(x^* F, \coker(\tilde\alpha_0))$. 
Thus
\begin{align*}
\tilde q\cdot \bar d\tilde \alpha  \cdot\tilde\iota
&= \tilde q\cdot\psi \cdot \bar d\alpha  \cdot \varphi^{-1}\cdot\tilde\iota \\
&= \bar\psi \cdot q\cdot \bar d\alpha  \cdot \iota\cdot \bar\varphi^{-1},
\end{align*}
which completes the proof.
\end{proof}

\begin{definition}
\label{def:id-pp}
The \emph{intrinsic differential} of the family of first-order principal parts $(x, \alpha)$ is the unique section
\begin{equation*}
\dd_x \alpha \in\Gamma(T, x^* \Omega_X\otimes\sheafHom_T(\ker\bar\alpha, \coker\bar\alpha))
\end{equation*}
such that, for each open subset $U \subseteq X$ over which $E$ and $F$ are free, the restriction of $\dd_x \alpha$ to $x^{-1}U\subseteq T$ coincides with the section of Proposition \ref{prop:id-pp-def} applied to the restriction (that is, the pullback) of $(x, \alpha)$ to $x^{-1}U$.
When $X$ is smooth over $k$, we will regard the intrinsic differential as an $\OO_T$-linear map
\begin{equation*}
\dd_x\alpha : x^* \T_X \to \sheafHom_T(\ker\bar\alpha, \coker\bar\alpha).
\end{equation*}
\end{definition}

The following result was used in the proof Proposition \ref{prop:id-pp-def}.

\begin{proposition}[Leibniz rule]
\label{prop:leibniz-pp}
Let $d : \OO_X\to \Omega_X$ be the universal derivation.
Let $\bar d : \PP_X^1\to \Omega_X$ be the unique $\OO_X$-linear map such that $\bar d \circ d_X^1 = d$.
Let $f,g \in \PP_X^1$ be sections defined over a common open subset of $X$, and let $\bar f, \bar g \in \OO_X$ be their respective images under the truncation map $\PP_X^1\to \OO_X$.
Then
\begin{equation}
\label{eqn:leibniz}
\bar d(fg) = \bar d f\cdot \bar g + \bar f \cdot \bar d g
\end{equation}
as sections of $\Omega_X$.
\end{proposition}

\begin{proof}
Both sides of (\ref{eqn:leibniz}) are $\OO_X$-bilinear functions of $f$ and $g$. As $\PP_X^1$ is generated as an $\OO_X$-module by the image of the universal differential operator $d_X^1 : \OO_X\to \PP_X^1$, we may assume that $f= d_X^1 u$ and $g= d_X^1 v$ for some $u,v\in \OO_X$. In this case, (\ref{eqn:leibniz}) reduces to the usual Leibniz rule satisfied by $d$.
\end{proof}

The next two remarks show that the intrinsic differential introduced in this section is compatible with the one from Definition \ref{def:id} and behaves well under pullback. 

\begin{remark}
\label{rmk:id-compat}
If $\tilde \alpha: E\to F$ is an $\OO_X$-linear map, then
\begin{equation*}
\dd_x\tilde \alpha = \dd_x (x^* d_E^1 \tilde\alpha). 
\end{equation*}
\end{remark}

\begin{remark}
\label{rmk:id-funct}
If $t: T'\to T$ be a morphism of schemes, then $\dd_{x\circ t}(t^* \alpha)$ is equal to the image of $t^*\dd_x \alpha$ under the $\OO_{T'}$-linear map
\begin{equation*}
\begin{tikzcd}
t^*(x^*\Omega_X\otimes \sheafHom_T(\ker \alpha_0,
\coker\alpha_0))
\ar[r,"\mu"] &
t^*x^*\Omega_X\otimes
\sheafHom_{T'}(\ker(t^*\alpha_0), \coker(t^*\alpha_0))
\end{tikzcd}
\end{equation*}
induced by the natural $\OO_{T'}$-linear map $t^*\ker\alpha_0 \to \ker(t^*\alpha_0)$ and isomorphsim $\coker(t^*\alpha_0)= t^*\coker(\alpha_0)$. 
In cases where $\mu$ is an isomorphism, which by Corollary \ref{coker-deg-loci} happens for example when $\coker(\bar\alpha)$ is a locally free $\OO_T$-module, we will abuse notation and write
\begin{equation*}
\dd_{x\circ t}(t^*\alpha) = t^*\dd_x\alpha.
\end{equation*}
\end{remark}

\section{Second-order differentials of jets}

Here we show that families of $2$-jets of morphisms from a smooth scheme over a field to affine space have natural second-order differentials.
We also discuss properties of symmetry and additivity that these second-order differentials inherit from the Hessian matrix of a function.
These properties will play an important role in the next section of this paper.

Let $k$ be a field.
Let $X$ be a smooth scheme over $k$.
Let $r$ be a positive integer.
Let $x : T\to X$ be a morphism of schemes and $f \in \Gamma(T,x^* (\PP_X^2)^{\oplus r})$ be a section, so that $(x,f)$ is a family of $2$-jets of morphisms from $X$ to $\A^r$ over $T$.
Let $f_1 \in \Gamma(T,x^* (\PP_X^1)^{\oplus r})$ be the truncation of $f$ to first order.

Let $\bar d : \PP_X^1 \to \Omega_X$ be the unique $\OO_X$-linear map such that $\bar d\circ d_X^1 = d$.
Let $\tilde d : \PP_X^2 \to \PP_X^1 \Omega_X$ be the unique $\OO_X$-linear map such that $\tilde d\circ d_X^2 = d_\Omega^1 \circ d$. The following diagram, where the vertical arrows are the natural truncation maps, commutes:
\begin{equation*}
\begin{tikzcd}
\PP_X^2 \ar[r,"\tilde d"] \ar[d, two heads] &
\PP_X^1 \Omega_X \ar[d, two heads] \\
\PP_X^1 \ar[r,"\bar d"] &
\Omega_X
\end{tikzcd}
\end{equation*}

We identify $\tilde d f\in \Gamma(T, x^* (\PP_X^1\Omega_X)^{\oplus r})$ with a global section of
\begin{equation*}
x^* \PP_X^1 \sheafHom_X(\T_X, \OO_X^{\oplus r})
\end{equation*}
and $\bar d f_1\in \Gamma(T, x^* \Omega_X^{\oplus r})$ with an $\OO_T$-linear map $x^* \T_X\to \OO_T^{\oplus r}$.
Let $K := \ker(\bar df_1)$ and $C := \coker(\bar df_1)$.

\begin{definition}
\label{def:2d-id-jet}
The intrinsic differential of the pair $(x, \tilde d f)$, viewed as a family of principal parts of order $1$ of $\OO_X$-linear maps over $T$, is an $\OO_T$-linear map 
\begin{equation*}
\dd_x(\tilde d f) : x^* \T_X \to \sheafHom_T(K,C).
\end{equation*}
We call the restriction
\begin{equation*}
\dd^2_x f : K \to \sheafHom_T(K,C)
\end{equation*}
of this $\OO_T$-linear map to $K\subseteq x^* \T_X$ the \emph{second intrinsic differential} of the family of $2$-jets $(x,f)$.
\end{definition}

The next two remarks show that Definition \ref{def:2d-id-jet} is compatible with Definition \ref{def:2nd-id} and behaves well under pullback. 

\begin{remark}
\label{rmk:2nd-id-compat}
If $\tilde f : X\to \A^r$ be a morphism of $k$-schemes, which we identify with an element of $\Gamma(X, \OO_X^{\oplus r})$, then 
\begin{align*}
\dd_x(\tilde d (x^* d_X^2 \tilde f))
= \dd_x(x^* \tilde d d_X^2 \tilde f)
= \dd_x( x^* d_\Omega^1 d \tilde f )
= \dd_x (d\tilde f).
\end{align*}
The last equality follows from Remark \ref{rmk:id-compat}.
\end{remark}

\begin{remark}
\label{rmk:2nd-id-funct}
Let $t : T'\to T$ be a morphism of schemes. If the cokernel of $\bar df_1 : x^*\T_X\to \OO_T^{\oplus r}$ is a locally free $\OO_T$-module, then
\begin{equation*}
\dd_{x\circ t}(\tilde d (t^*f)) = 
\dd_{x\circ t}(t^* \tilde d f) = 
t^*\dd_x(\tilde d f)
\end{equation*}
as elements of
\begin{equation*}
\Hom_{T'}(t^*x^* \T_X, 
\sheafHom_{T'}( \ker(\bar d(t^* f_1)), \coker(\bar d(t^* f_1)) ) ).
\end{equation*}
In general, $\dd_x(\tilde d (t^* f))$ equals the image of $t^*\dd_x(\tilde d f)$ under a natural map, see Remark \ref{rmk:id-funct}.
\end{remark}

We now turn to properties of symmetry and additivity that the second-order differentials $\dd_x(\tilde df)$ and $\dd_x^2 f$ inherit from the Hessian matrix of a function.
We begin with symmetry: Proposition \ref{prop:2nd-id-symm} below extends Remark \ref{rmk:2nd-id-sym} to the family of $2$-jets $(x, f)$.
As in section \ref{sec:2nd-order-sings}, given a $k$-scheme $S$ and an $\OO_S$-module $M$, we write 
\begin{equation*}
\Box^2 M :=
\begin{cases}
\Sym^2 M &\text{if $\charac(k)\ne 2$}\\
\wedge^2 M &\text{if $\charac(k)=2$}.
\end{cases}
\end{equation*}
Furthermore, given a submodule $A\subseteq M$ we denote by $M \boxempty A$ the image of $M\otimes A$ under the quotient map $M\otimes M \to \Box^2 M$.
Let
\begin{align*}
\theta : \sheafHom_T(x^*\T_X \otimes K, C) &\xrightarrow\sim \sheafHom_T(x^*\T_X ,\sheafHom_T(K,C))\\
\bar \theta : \sheafHom_T(K \otimes K, C) &\xrightarrow\sim \sheafHom_T(K,\sheafHom_T(K,C))
\end{align*}
be the natural $\OO_T$-linear isomorphisms, both of which are described by the rule $b\mapsto (v\mapsto b(v\otimes -))$.

\begin{proposition}
\label{prop:2nd-id-symm}
The second-order differentials $\dd(\tilde df)$ and $\dd^2 f$ are symmetric in characteristic different from $2$ and in skew-symmetric in characteristic $2$. More precisely:
\begin{align*}
\theta^{-1} (\dd(df))&\in \Hom_T(x^* \T_X \boxempty K, C)\\
\bar\theta^{-1}(\dd^2 f) &\in \Hom_T(\Box^2 K,C)
\end{align*}
\end{proposition}

\begin{proof}
The question being Zariski-local on $T$, we may assume that there exist global sections $x_1,\dotsc, x_n\in \Gamma(X,\OO_X)$ whose differentials from a basis for $\Omega_X$ as an $\OO_X$-module. Then the result follows from the combination of Lemma \ref{lemma:2nd-id-locally} and the first part of Lemma \ref{lemma:hessian-pp} below.
\end{proof}

Next we consider an additivity property of second-order differentials. 
The case of a single 2-jet, where $T=\Spec k$ and $x : T\to X$ is a rational point, is clearest. 

\begin{example}
  \label{rmk:2nd-id-additivity}
Suppose that $T=\Spec k$ and $x : T\to X$ is a rational point, so that
\begin{align*}
f&\in (\OO_X/\mathfrak m_x^3)^{\oplus r} \\
\tilde d f &\in \sheafHom_X(\T_X, \OO_X^{\oplus r})\otimes \OO_X/\mathfrak m_x^2\\
\bar d f_1 &\in \Hom_k(\T_X(x), k^{\oplus r}).
\end{align*}
Let 
\begin{align*}
\Hess (f) : \T_X(x) \times \T_X(x) \to k^{\oplus r}
\end{align*}
be the \emph{Hessian bilinear map} defined with respect to a choice of \'etale coordinates near $x\in X(k)$.
By Remark \ref{rmk:2nd-id-locally}, we have a commutative diagram as follows:
\begin{equation*}
\begin{tikzcd}
T_X(x)\otimes \T_X(x) \ar[r, "\Hess (f)"] &
k^{\oplus r} \ar[d, two heads]\\
\T_X(x) \otimes K \ar[u, hook] \ar[r, "\dd_x(\tilde df)"] &
C
\end{tikzcd}
\end{equation*}
Let
\begin{equation*}
\begin{tikzcd}
0 \ar[r] &
\Sym^2(\mathfrak m_x/\mathfrak m_x^2) \ar[r, "\iota"] &
\OO_X/\mathfrak m_x^3 \ar[r] &
\OO_X/\mathfrak m_x^2 \ar[r] & 
0
\end{tikzcd}
\end{equation*}
be the canonical short exact sequence.
Let 
\begin{equation*}
\beta : \Sym^2(\mathfrak m_x/\mathfrak m_x^2) \to \Hom(\T_X(x)\otimes \T_X(x), k)
\end{equation*}
be the unique $k$-linear map such that
\begin{equation*}
\beta(uv) = u\otimes v + b\otimes v
\end{equation*}
for all $u, v\in \mathfrak m_x/\mathfrak m_x^2 = \T_X(x)^\vee$. 
Finally, let $\delta \in \Sym^2(\mathfrak m_x/\mathfrak m_x^2)$.
Then 
\begin{align*}
\Hess (f + \iota \delta)
&= \Hess (f) + \Hess (\iota \delta) \\
&= \Hess (f) + \beta(\delta)
\end{align*}
as $k$-linear maps $\T_X(x)\otimes \T_X(x)\to k^{\oplus r}$. 
It follows that
\begin{equation*}
\dd_x(\tilde d(f+\iota\delta)) = \dd_x(\tilde df) + \overline{\beta(\delta)}
\end{equation*}
as $k$-linear maps $\T_X(x)\otimes K \to C$, where $\overline{\beta(\delta)}$ denotes the natural image of $\beta(\delta)$ in $\Hom_k(\T_x(x)\otimes K, C)$.
\end{example}

Returning to the general case, let
\begin{equation*}
\begin{tikzcd}
0 \ar[r] & 
\Sym^2 \Omega_X \ar[r,"\iota"] &
\PP_X^2 \ar[r] &
\PP_X^1\ar[r] &
0
\end{tikzcd}
\end{equation*}
be the canonical short exact sequence.
Let
\begin{equation*}
\beta : \Sym^2 \Omega_X\to (\T_X\otimes \T_X)^\vee
\end{equation*}
be the unique $\OO_X$-linear map such that $\beta(uv) = u\otimes v + v\otimes u$ for all $u,v\in \Omega_X = \T_X^\vee$.
Let $\bar \beta$ be the composition of the following $\OO_T$-linear maps:
\begin{equation*}
\begin{tikzcd}
x^*(\Sym^2 \Omega_X)^{\oplus r} \ar[r, "\beta"] &
\sheafHom_T(x^* \T_X\otimes x^* \T_X, \OO_T^{\oplus r}) \ar[ld, two heads]\\
\sheafHom_T(x^* \T_X \otimes K, C) \ar[r, "\theta", "\sim"'] &
\sheafHom_T(x^*\T_X ,\sheafHom_T(K,C))
\end{tikzcd}
\end{equation*}

\begin{remark}
\label{rmk:quadratic-to-bilinear}
The image of $\beta$ is $(\Box^2 \T_X)^\vee$, so that of $\bar \beta$ is
\begin{equation*}
\theta(\sheafHom_T(x^* \T_X \boxempty K, C)).
\end{equation*}
\end{remark}

\begin{proposition}
\label{prop:2nd-id-equiv}
Let $\delta\in \Gamma(T,x^* (\Sym^2 \Omega_X)^{\oplus r})$ be a section. Then
\begin{equation*}
\dd_x(\tilde d(f+\iota \delta)) = \dd_x(\tilde d f) + \bar \beta \delta
\end{equation*}
as elements of $\Hom_T(x^*\T_X ,\sheafHom_T(K,C))$, and
\begin{equation*}
\dd_x^2 (f+\iota \delta) = \dd_x^2 f + (\bar \beta \delta)|_K
\end{equation*}
as elements of $\Hom_T(K,\sheafHom_T(K,C))$.
\end{proposition}

\begin{proof}
The question being Zariski-local on $T$, we may assume that there exist global sections $x_1,\dotsc, x_n\in \Gamma(X,\OO_X)$ whose differentials from a basis for $\Omega_X$ as an $\OO_X$-module. Then the result follows from the combination of Lemma \ref{lemma:2nd-id-locally} and the second part of Lemma \ref{lemma:hessian-pp} below.
\end{proof}

In the remainder of this section we prove the two lemmas that were used in the proofs of Propositions \ref{prop:2nd-id-symm} and \ref{prop:2nd-id-equiv}.

\begin{setup}
\label{setup:hessian}
Suppose that there exist sections $x_1,\dotsc, x_n\in \Gamma(X,\OO_X)$ be sections whose differentials form an $\OO_X$-linear basis for $\Omega_X$.
Fix a choice of such sections.
Let $\{\partial_1, \dotsc, \partial_n \}$ be the basis of $\T_X= \Omega_X^\vee$ that is dual to the basis $\{dx_1,\dotsc, dx_n\}$ of $\Omega_X$. 
Let
$\Hess  : \OO_X\to \Omega_X\otimes \Omega_X$
be the second-order $k$-linear differential operator differential operator that sends 
\begin{equation*}
\varphi\mapsto \sum_{a,b=1}^n
\partial_a \partial_b \varphi\cdot
dx_a \otimes dx_b.
\end{equation*}
For each $a=1,\dotsc, n$, we identify $\partial_a\in \Gamma(X,\T_X)$ with a derivation $\partial_a : \OO_X\to \OO_X$, and let
$\bar\partial_a : \PP_X^1\to \OO_X$
and
$\tilde\partial_a : \PP_X^2 \to \PP_X^1$
be the $\OO_X$-linear maps characterized by the conditions that $\bar\partial_a \circ d_X^1 = \partial_a$ and $\tilde \partial_a \circ d_X^2 = d_X^1 \circ \partial_a$.
Let
$\overline{\Hess } : \PP_X^2\to \Omega_X\otimes\Omega_X$
be the unique $\OO_X$-linear such that $\overline{\Hess }\circ d_X^2 = \Hess $.
\end{setup}

\begin{lemma}
\label{lemma:2nd-id-locally}
Assume Setup \ref{setup:hessian}.
Let $\tilde\theta : \Omega_X\otimes\Omega_X \to \sheafHom_X(\T_X,(\T_X)^\vee)$ be the $\OO_X$-linear map that sends $b\mapsto (v\mapsto b(v\otimes-))$. 
The intrinsic differential $\dd_x(\tilde df)$ is equal to the composition of the $\OO_T$-linear maps
\begin{equation*}
\begin{tikzcd}
(x^* \T_X)\ar[r,"\tilde \theta(\overline{\Hess } f)"] &
\sheafHom_T(x^* \T_X, \OO_T^{\oplus r}) \ar[r,two heads] &
\sheafHom_T(\ker(\bar f_1),\coker(\bar df_1)),
\end{tikzcd}
\end{equation*}
where the second arrow is induced by the inclusion $\ker(d\bar f_1)\hookrightarrow x^* \T_X$ and the projection $\OO_T^{\oplus r} \to \coker(d\bar f_1)$.
\end{lemma}

\begin{proof}
Write $H := \sheafHom_X(\T_X, \OO_X^{\oplus r})$. 
Let $\nabla : H\to \Omega_X\otimes H$ be the $k$-linear map given by differentiation of matrix elements with respect to the basis $\{\partial_1,\dotsc, \partial_n\}\subseteq \Gamma(X, \T_X)$ of $\T_X$ and the standard basis of $\OO_X^{\oplus r}$. 
Let $\overline\nabla : \PP_X^1 H \to \Omega_X\otimes H$ 
be the unique $\OO_X$-linear map such that $\overline\nabla \circ d_H^1 = \nabla$. 
Then $\overline\nabla \circ \tilde d = \overline{\Hess }$ as $\OO_X$-linear maps $(\PP_X^2)^{\oplus r} \to \Omega_X\otimes H = (\Omega_X\otimes \Omega_X)^{\oplus r}$. Indeed, $\PP_X^2$ is generated as an $\OO_X$-module by the image of $d_X^2 : \OO_X\to \PP_X^2$ and, for each local section $\varphi\in \OO_X^{\oplus r}$, we have
\begin{align*}
\overline\nabla \tilde d d_X^2 \varphi
&= \overline\nabla d_H^1 d\varphi \\
&= \nabla(d\varphi)\\
&= \textstyle  \nabla( \sum_b \partial_b \varphi\cdot dx_b ) \\
&= \textstyle \sum_b d(\partial_b \varphi) \otimes dx_b \\
&= \textstyle \sum_{a,b} \partial_a \partial_b \varphi \cdot dx_a\otimes dx_b \\
&= \operatorname{Hess}(\varphi)\\
&= \overline{\Hess }(d_X^2 \varphi).
\end{align*} 
In particular, $\overline\nabla(\tilde d f) = \overline{\Hess }(f)$ as elements of $\Gamma(T, x^*(\Omega_X\otimes H))$. 
The result now follows from the definition of the intrinsic differential (Defintion \ref{def:id-pp}).
\end{proof}

\begin{lemma}
\label{lemma:hessian-pp}
Assume Setup \ref{setup:hessian}.
The $\OO_X$-linear map
\begin{equation*}
\overline{\Hess } : \PP_X^2 \to \Omega_X\otimes \Omega_X
\end{equation*}
factors through $\im(\beta) = (\Box^2 \T_X)^\vee$.
Furthermore, 
\begin{equation*}
\overline{\Hess }\circ \iota=\beta
\end{equation*}
as $\OO_X$-linear maps $\Sym^2 \Omega_X\to \Omega_X\otimes \Omega_X$. 
\end{lemma}

\begin{proof}
First, we claim that $\PP_X^2$ is freely generated as an $\OO_X$-module by the sections
\begin{equation*}
d_X^2(x_i x_j), d_X^2 x_i, d_X^2 1\in \Gamma(X,\PP_X^2),
\end{equation*}
where $i,j=1,\dotsc, n$.
To see this, let $\A_k^n$ denote the affine space over $k$ with coordinates $t_1,\dotsc, t_n$, and let $\varphi : X\to \A^n_k$ be the unique map of $k$-schemes such that $\varphi^\# t_i = x_i$ for all $i=1,\dotsc, n$.
Then $\varphi$ is \'etale, so induces an isomorphism of $\OO_X$-algebras $\varphi^* \PP_{\A^n}^2 \xrightarrow\sim \PP_X^2$.
This isomorphism sends $d_{\A^n}^2 t_i\mapsto d_X^2 x_i$ for all $i=1,\dotsc, n$.
Now let $\varepsilon_1,\dotsc, \varepsilon_n$ be indeterminates.
It is a standard fact that the unique map of $\OO_{\A^n}$-algebras 
\begin{equation*}
\OO_{\A^n}[\varepsilon_1,\dotsc, \varepsilon_n]\to \PP_{\A^n}^2
\end{equation*}
that sends $\varepsilon_i \mapsto d_{\A^n}^2 t_i -t_i$ for all $i=1,\dotsc, n$ is surjective with kernel $\langle \varepsilon_1,\dotsc, \varepsilon_n\rangle^2$.
The claim follows.

Next, we note that
\begin{align*}
\overline{\Hess }( d^2_X(x_i x_j))
&= \sum_{a,b} \partial_a \partial_b(x_i x_j) dx_a \otimes dx_b \\
% &= \sum_{a,b} \partial_a (x_j \partial_b x_i + x_i \partial_b x_j) dx_a \otimes dx_b\\
% &= \sum_a \partial_a x_j dx_a \otimes dx_i + \partial_a x_i dx_a\otimes dx_j \\
&= dx_j \otimes dx_i + dx_i \otimes dx_j.
\end{align*}
Similarly, $\overline{\Hess }(d_X^2 x_i) = 0$ for all $i=1,\dotsc,n$, and $\overline{\Hess }( d_X^2 1 )=0$.
Furthermore, 
\begin{align*}
\overline{\Hess }( \iota (dx_i\cdot dx_j ))
&= \overline{\Hess }((d_X^2 x_i -x_i)(d_X^2 x_j -x_j))\\
&= \overline{\Hess }( d_X^2 (x_i x_j) - x_i d_X^2 x_j 
-x_j d_X^2 x_i + x_i x_j )\\
&= dx_j \otimes dx_i + dx_i \otimes dx_j.
\end{align*}
The result follows.
\end{proof}

\section{Universal second-order singularities}

The proofs of Theorems \ref{thm:intro-smoothness} and \ref{thm:intro-2nd-sings} outlined in section \ref{sec:generic-maps} made use of schemes of singularities inside the second jet scheme. 
In this section we construct these schemes of singularities and reduce the facts about their codimension and smoothness that we used in section \ref{sec:generic-maps}, namely Theorems \ref{thm:univ-bad-locus} and \ref{thm:univ-2nd-sings} below, to corresponding facts about schemes of linear maps satisfying certain symmetry and rank conditions.
We will prove the latter in sections \ref{sec:sym-bil} and \ref{sec:min-codim} below. 

Let $k$ be a field.
Let $X$ be a smooth scheme of pure dimension $n$ over $k$. 
Let $\A^r$ be the affine space of dimension $r$ over $k$.
For $m=1, 2$, let $J^m := J^m(X,\A^r)$ be the $m$th jet scheme, see Definition \ref{def:jet-scheme}, and let $\tau_m \in \Gamma(J^m, (\PP_X)_{J^m}^{\oplus r})$ be the tautological section, see Definition \ref{def:taut-sect}; then the pair $(J^m \to X, \tau_m)$ is a universal family of $m$-jets of morphisms from $X$ to $\A^r$. 

Let $q : J^2\to J^1$ be the morphism of vector bundles over $X$ induced by the $\OO_X$-linear truncation map $\varepsilon : \PP_X^2\to \PP_X^1$. Thus $q$ is the unique morphism of $X$-schemes such that $q^*\tau_1 = \varepsilon(\tau_2)$. 
Let $\bar d : \PP_X^1 \to \Omega_X$ be the unique $\OO_X$-linear map such that $\bar d \circ d_X^1 = d$.
Let $\tilde d : \PP_X^2 \to \PP_X^1 \Omega_X$ be the unique $\OO_X$-linear map such that $\tilde d \circ d_X^2 = d_\Omega^1 \circ d$. 
Then the image of the differential
\begin{equation*}
\tilde d \tau_2 \in \Gamma(J^2, (\PP_X^1 \Omega_X^{\oplus r})_{J^2})
\end{equation*}
under the truncation map $\PP_X^1 \Omega_X\to \Omega_X$ is equal to
\begin{equation*}
q^* \bar d \tau_1 \in \Gamma(J^2, (\Omega_X^{\oplus r})_{J^2}).
\end{equation*}

Let $i$ be a nonnegative integer.
Let $\Sigma^i\subseteq J^1$ be the $i$th universal critical locus, see Definition \ref{def:univ-crit-loci}.
Let $\bar\tau_1 \in \Gamma(\Sigma^i, (\PP_X^1)_{\Sigma^i}^{\oplus r})$ be the restriction of $\tau_1$ to $\Sigma^i$.
Let $K$ and $C$ respectively denote the kernel and cokernel of $\bar d\bar \tau_1 : (\T_X)_{\Sigma^i} \to \OO_{\Sigma^i}^{\oplus r}$. 
Then $K$ and $C$ are locally free $\OO_{\Sigma^i}$-modules, and
\begin{equation*}
q^* K = \ker(q^*\bar d\bar \tau_1)
\qquad\text{and}\qquad
q^* C = \coker(q^*\bar d\bar \tau_1)
\end{equation*}
by Corollary \ref{coker-deg-loci}. 

Let
$
\bar \tau_2 \in \Gamma(q^{-1}\Sigma^i, (\PP_X^2)_{q^{-1}\Sigma^i}^{\oplus r})
$
be the restriction of $\tau_2$ to $q^{-1}\Sigma^i$.
The intrinsic differential of the family of first-order principal parts $(q^{-1}\Sigma^i, \tilde d\bar \tau_2)$ is an $\OO_{q^{-1}\Sigma^i}$-linear map
\begin{equation*}
\dd_{q^{-1}\Sigma^i}(\tilde d\bar \tau_2)
: (\T_X)_{q^{-1}\Sigma^i} \to
\sheafHom_{q^{-1}\Sigma^i}( q^* K, q^* C).
\end{equation*}
The second intrinsic differential of the family of $2$-jets $(q^{-1}\Sigma^i, \bar \tau_2)$ is the restriction 
\begin{equation*}
\dd^2_{q^{-1}\Sigma^i} \bar \tau_2
: q^* K \to
\sheafHom_{q^{-1}\Sigma^i}( q^* K, q^* C)
\end{equation*}
of this map to $q^* K$.

Let $j$ be a nonnegative integer.

\begin{definition}
\label{def:univ-2nd-sings}
The \emph{universal bad locus} $B^i$ is the closed subscheme of $q^{-1}\Sigma^i$ defined as follows.
If $n\ge i(|n-r|+i)$, then $B^i(f)$ is the first degeneracy locus
\begin{equation*}
B^i := \Sigma^1( \dd_{q^{-1}\Sigma^i}(\tilde d\bar \tau_2) ) \subseteq q^{-1}\Sigma^i.
\end{equation*}
Otherwise, $B^i = q^{-1}\Sigma^i$.
The locally closed subscheme
\begin{equation*}
\Sigma^{i,j} :=
\Sigma^j( \dd^2_{q^{-1}\Sigma^i} \bar \tau_2 )
\setminus
\Sigma^{j+1}( \dd^2_{q^{-1}\Sigma^i} \bar \tau_2 )
\subseteq q^{-1}\Sigma^i
\end{equation*}
is called the \emph{universal locus of second-order singularities with symbol $(i,j)$} in $J^2(X,\A^r)$.
\end{definition}

The next result relates Definitions \ref{def:univ-2nd-sings} and \ref{def:2nd-order-sings}.

\begin{proposition}
\label{prop:univ-sings-def-2}
Let $U\subseteq X$ be an open subset and let $f : U\to \A^r$ be a morphism of $k$-schemes.
Let $j^2 f : U\to J^1(X, \A^r)$ be the $2$-jet of $f$.
Then
\begin{equation*}
(j^2 f)^{-1}B^i = B^i(f)
\qquad\text{and}\qquad 
(j^2 f)^{-1}\Sigma^{i,j} = \Sigma^{i,j}(f).
\end{equation*}
as subschemes of $U$.
\end{proposition}

\begin{proof}
By Proposition \ref{prop:univ-sings-def-1} and the observation that $q\circ j^2 f = j^1 f$, 
\begin{equation*}
(j^2 f)^{-1} q^{-1}\Sigma^i = \Sigma^i(f).  
\end{equation*}
By Remarks \ref{rmk:2nd-id-compat} and \ref{rmk:2nd-id-funct},
\begin{align*}
(j^2 f|_{\Sigma^i(f)})^*
\dd_{q^{-1}\Sigma^i} (\tilde d\bar \tau_2)
&= \dd_{\Sigma^i(f)}(\tilde d ( j^2 f|_{\Sigma^i(f)})^* \tau_2 )\\
&= \dd_{\Sigma^i(f)}( \tilde d (d_X^2 f|_{\Sigma^i(f)} ) )\\
&= \dd_{\Sigma^i(f)}(df).
\end{align*}
Thus, if $n\ge i(|n-r|+i)$, then
\begin{align*}
(j^2 f)^{-1} B^i
&= (j^2 f|_{\Sigma^i(f)})^{-1} \Sigma^1( \dd_{q^{-1}\Sigma^i} (\tilde d\bar \tau_2) ) \\
&=  \Sigma^1( (j^2 f|_{\Sigma^i(f)})^* \dd_{q^{-1}\Sigma^i} (\tilde d\bar \tau_2) ) \\
&= \Sigma^1( \dd_{\Sigma^i(f)}(df) ) \\
&= B^i(f).
\end{align*}
If $n<i(|n-r|+i)$, then by definition $B^i = q^{-1}\Sigma^i$ and $B^i(f) = \Sigma^i(f)$, so
\begin{equation*}
(j^2 f)^{-1}B^i = B^i(f)
\end{equation*}
in this case also.
In general, 
\begin{equation*}
(j^2 f|_{\Sigma^i(f)})^*
\dd^2_{q^{-1}\Sigma^i} \bar \tau_2 =
\dd^2_{\Sigma^i(f)} f
\end{equation*}
and therefore
\begin{equation*}
(j^2 f)^{-1} \Sigma^{i,j} = \Sigma^{i,j}(f). \qedhere
\end{equation*}
\end{proof}

\begin{theorem}
\label{thm:univ-bad-locus}
Let $m :=\min(n,r)$. 
The universal bad locus $B^i\subseteq q^{-1}\Sigma^i$ is nonempty if, and only if, $1\le i\le m$.

If $n<i(|n-r|+i)$, then by definition $B^i=q^{-1}\Sigma^i$, and $B^i$ is smooth over $X$ of pure relative codimension $i(|n-r|+i)$ in $J^2(X,\A^r)$.  

Suppose that $n\ge i(|n-r|+i)$.
If $\charac(k)\ne 2$, then $B^i \subseteq q^{-1}\Sigma^i$ has relative codimension $n+1$ in $J^2(X,\A^r)$ over $X$.
The same holds if $\charac(k)=2$, with two exceptions:
\begin{enumerate}
\item The case where $i=1$ and $r\ge n$.
\item The case where $i=1$, $r=1$ and $n$ is odd.
\end{enumerate}
In both cases, the universal bad locus $B^i$ has relative codimension $n$ in $J^2(X,\A^r)$ over $X$.
\end{theorem}

\begin{proof}
By Proposition \ref{prop:univ-crit-loci}, the universal critical locus $\Sigma^i\subseteq J^1$ is nonempty if, and only if, $0\le i\le m$, so we may assume these inequalities hold.
Then $\Sigma^i$ is smooth over $X$ and of pure relative codimension $i(|n-r|+i)$ in $J^1$ over $X$, again by Proposition \ref{prop:univ-crit-loci}.

Recall that $J^m := \V((\PP_X^m)^{\oplus r})$ for $m=1,2$, and that $q : J^2\to J^1$ the morphism of vector bundles over $X$ induced by the truncation map $\varepsilon : \PP_X^2\to \PP_X^1$.
This truncation map is surjective, so $q$ is smooth and surjective.
Therefore $q^{-1}\Sigma^i$ is smooth over over $X$, and of pure relative codimension $i(|n-r|+i)$ in $J^2$ over $X$.
In particular, the result holds if $n<i(|n-r|+i)$.

Suppose that $n\ge i(|n-r|+i)$. 
Let
\begin{equation*}
\Box^2 \T_X :=
\begin{cases}
\Sym^2 \T_X & \text{if }\charac(k)\ne 2\\
\wedge^2 \T_X & \text{if }\charac(k)= 2.
\end{cases}
\end{equation*}
Let $(\T_X)_{\Sigma^i} \boxempty K$ be the image of $(\T_X)_{\Sigma^i} \otimes K$ under the quotient map
\begin{equation*}
(\T_X\otimes \T_X)_{\Sigma^i} \twoheadrightarrow (\Box^2 \T_X)_{\Sigma^i}.
\end{equation*}
By Lemma \ref{lemma:box-free} below, $(\T_X)_{\Sigma^i} \boxempty K$ is a locally free $\OO_{\Sigma^i}$-module.

Let
\begin{equation*}
\theta : \sheafHom_{\Sigma^i}( (\T_X)_{\Sigma^i}\otimes K, C)
\xrightarrow \sim 
\sheafHom_{\Sigma^i}( (\T_X)_{\Sigma^i},
\sheafHom_{\Sigma^i}( K, C))
\end{equation*}
be the natural $\OO_{\Sigma^i}$-linear isomorphism that sends $b\mapsto (v\mapsto b(v\otimes -))$.
By Proposition \ref{prop:2nd-id-symm}, the inverse image $\theta^{-1}(\dd_{q^{-1} \Sigma^i}(\tilde d\bar \tau_2))$ is contained the $\OO_{q^{-1}\Sigma^i}$-submodule
\begin{equation*}
q^* \sheafHom_{\Sigma^i}( (\T_X)_{\Sigma^i}\boxempty K, C)\subseteq
q^* \sheafHom_{\Sigma^i}( (\T_X)_{\Sigma^i}\otimes K, C).
\end{equation*}

Let $H\to \Sigma^i$ be the vector bundle associated with the locally free $\OO_{\Sigma^i}$-module $\sheafHom_{\Sigma^i}( (\T_X)_{\Sigma^i}\boxempty K, C)$. Let $h : ((\T_X)_{\Sigma^i}\boxempty K)_H\to C_H$ be the tautological $\OO_H$-linear map.
Let $D : q^{-1}\Sigma^i\to H$ be the unique morphism of $\Sigma^i$-schemes such that
\begin{equation*}
D^* h = \theta^{-1}(\dd_{q^{-1} \Sigma^i}(\tilde d\bar \tau_2)).
\end{equation*}
The argument that follows is based on the following commutative diagram with Cartesian squares:
\begin{equation*}
\begin{tikzcd}
& B^i\ar[d, hook] \ar[r] \ar[rd, phantom, "\square"] &
\Sigma^1(\theta h) \ar[d, hook]
\\
J^2 \ar[d, "q"'] & \ar[l, hook] q^{-1}\Sigma^i \ar[r, "D"'] \ar[d] \ar[ld, phantom, "\square"] &
H \ar[ld, bend left = 10]
\\
J^1 &\ar[l,hook] \Sigma^i
\end{tikzcd}
\end{equation*}

The respective ranks of $K$ and $C$ as $\OO_{\Sigma^i}$-modules are $n-m+i$ and $r-m+i$.
If $i=0$, then $\sheafHom_{\Sigma^i}(K,C)=0$, so $\Sigma^1(\theta h)$ and $B^i$ are empty. 

Suppose $i>0$.
By Proposition \ref{prop:sym-deg-loci-2} below, if $\charac(k)\ne 2$, then the first degeneracy locus $\Sigma^1(\theta h)$ has relative codimension
\begin{equation*}
n - (n-m+i)(r-m+i) + 1 = n - i(|n-r|+i) + 1
\end{equation*} 
in $H$ over $\Sigma^i$.
Furthermore, if $\charac(k)=2$, the same holds provided that
\begin{enumerate}
\item[(a)] $n-m+i > 1$; and
\item[(b)] if $r-m+i=1$, then $n-m+i=n$ and $n$ is even.
\end{enumerate}
Otherwise, $\Sigma^1(\theta h)$ has relative codimension $n - i(|n-r|+i)$ in $H$ over $\Sigma^i$.

Conditions (a) and (b) may be respectively rephrased as follows:
\begin{enumerate}
\item[(a')] If $i=1$, then $r<n$.
\item[(b')] If $i=1$ and $r\le n$, then $r=1$ and $n$ is even.
\end{enumerate}
Thus (a) and (b) hold if, and only if, conditions (1), (2) and (3) from the statement of the theorem are satisfied.

To prove the result it suffices to show that the relative codimension of $B^i$ in $q^{-1}\Sigma^i$ over $\Sigma^i$ is equal to the relative codimension of $\Sigma^1(\theta h)$ in $H$ over $\Sigma^i$.
We will do this by showing that $D$ is smooth and surjective.

Let $G$ denote the vector bundle
\begin{equation*}
G := \V( (\Sym^2 \Omega_X)^{\oplus r})
\end{equation*}
regarded as an additive group scheme over $X$.
The map $\iota$ in the canonical short exact sequence
\begin{equation*}
\begin{tikzcd}
0 \ar[r] & 
\Sym^2 \Omega_X \ar[r,"\iota"] &
\PP_X^2 \ar[r,"\varepsilon"] &
\PP_X^1\ar[r] &
0
\end{tikzcd}
\end{equation*}
induces an action of the additive group $J^1$-scheme $G_{J^1} := G\times_X J^1$ by translations on $J^2$ that gives $J^2$ the structure of a principal $G_{J^1}$-bundle over $J^1$.
This action restricts to an action of the additive group $\Sigma^i$-scheme $G_{\Sigma^i} := G\times_{J^1}\Sigma^i$ on $q^{-1}\Sigma^i$ that gives $q^{-1}\Sigma^i$ the structure of a principal $G_{\Sigma^i}$-bundle over $\Sigma^i$.

Let
\begin{equation*}
(\Sym^2 \Omega_X)^{\oplus r}_{\Sigma^i} \to
\sheafHom_{\Sigma^i}( (\T_X)_{\Sigma^i}\boxempty K, C).
\end{equation*}
be the natural $\OO_{\Sigma^i}$-linear surjection, see Remark \ref{rmk:quadratic-to-bilinear}.
Let $\varphi : G_{\Sigma^i}\to H$ be the corresponding map of vector bundles over $\Sigma^i$, which we regard as a homomorphism of additive group $\Sigma^i$-schemes.
Then $D : q^{-1}\Sigma^i\to H$ is $\varphi$-equivariant by Proposition \ref{prop:2nd-id-equiv}.
This implies that $D$ is smooth and surjective, since $\varphi$ is smooth and surjective.
\end{proof}

\begin{theorem}
\label{thm:univ-2nd-sings}
The universal locus of second-order singularities $\Sigma^{i,j}$ is nonempty if, and only if,
\begin{enumerate}
\item $i \le m$; and 
\item $j \le n-m+i$; and
\item if $i=0$, then $j=0$; and
\item if $\charac(k)=2$, $i=1$ and $r\le n$, then $n-m+i-j$ is even.
\end{enumerate}
In this case, $\Sigma^{i,j}$ is smooth over $X$ and of pure relative codimension
\begin{equation*}
\label{rel-codim-cases}
i(|n-r|+i) +
j(n-m+i-j)(r-m+i-1) +\tfrac 1 2 j(j\pm1)(r-m+i)
\end{equation*}
in $J^2(X,\A^r)$ over $X$.
The symbol $\pm$ appearing in this expression should be read as ``plus'' if $\charac(k)\ne 2$ and as ``minus'' if $\charac(k)=2$.
\end{theorem}

\begin{proof}
Let
\begin{equation*}
\Box^2 K :=
\begin{cases}
\Sym^2 K & \text{if }\charac(k)\ne 2\\
\wedge^2 K & \text{if }\charac(k)= 2.
\end{cases}
\end{equation*}
Let
\begin{equation*}
\theta : \sheafHom_{\Sigma^i}( K\otimes K, C)
\xrightarrow \sim 
\sheafHom_{\Sigma^i}( K,
\sheafHom_{\Sigma^i}( K, C))
\end{equation*}
be the natural $\OO_{\Sigma^i}$-linear isomorphism that sends $b\mapsto (v\mapsto b(v\otimes -))$.
By Proposition \ref{prop:2nd-id-symm}, the inverse image $\theta^{-1}(\dd^2_{q^{-1} \Sigma^i}\bar\tau_2)$ is contained the $\OO_{q^{-1}\Sigma^i}$-submodule
\begin{equation*}
q^* \sheafHom_{\Sigma^i}( \Box^2 K, C)\subseteq
q^* \sheafHom_{\Sigma^i}( K\otimes K, C).
\end{equation*}

Let $H\to \Sigma^i$ be the vector bundle associated with the locally free $\OO_{\Sigma^i}$-module $\sheafHom_{\Sigma^i}( \Box^2 K, C)$. Let $h : (\wedge^2 K)_H\to C_H$ be the tautological $\OO_H$-linear map.
Let $D : q^{-1}\Sigma^i\to H$ be the unique morphism of $\Sigma^i$-schemes such that
\begin{equation*}
D^* h = \theta^{-1}(\dd^2_{q^{-1} \Sigma^i}\bar\tau_2).
\end{equation*}
Consider the following commutative diagram with Cartesian squares:
\begin{equation*}
\begin{tikzcd}
& \Sigma^{i,j} \ar[d, hook] \ar[r] \ar[rd, phantom, "\square"] &
\Sigma^j(\theta h)\setminus \Sigma^{j+1}(\theta h) \ar[d, hook]
\\
J^2 \ar[d, "q"'] & \ar[l, hook] q^{-1}\Sigma^i \ar[r, "D"'] \ar[d] \ar[ld, phantom, "\square"] &
H \ar[ld, bend left = 10]
\\
J^1 &\ar[l,hook] \Sigma^i
\end{tikzcd}
\end{equation*}

As in the proof of Theorem \ref{thm:univ-bad-locus}, the $\Sigma^i$-scheme $q^{-1}\Sigma^i$ is smooth over $X$, is of pure relative codimension $i(|n-r|+i)$ in $J^2$ over $X$, and the morphism $D$ is smooth and surjective.
Thus it suffices to show that $H$ is nonempty if, and only if, (1)--(4) hold, and smooth of relative codimension 
\begin{equation*}
j(n-m+i-j)(r-m+i-1) +\tfrac 1 2 j(j\pm1)(r-m+i)
\end{equation*}
in $H$ over $\Sigma^i$ when these conditions hold.
If $i>0$, this follows from Proposition \ref{prop:sym-deg-loci-1} below and the observation that the ranks of the locally free $\OO_{\Sigma^i}$-modules $K$ and $C$ are $n-m+i$ and $r-m+i$, respectively.
If $i=0$, then $\sheafHom_{\Sigma^i}(K,C)=0$, so the subscheme $\Sigma^j(\theta h)\setminus \Sigma^{j+1}(\theta h) \subseteq H$ is empty if $j>0$ and equal to $H$ if $j=0$.
\end{proof}

\section{Geometry on Grassmannians}

In this section we collect a few facts about Grassmannians that we will use in the proof of Theorem \ref{thm:univ-deg-bil} below.
Here we review the standard affine charts on Grassmannians and special Schubert cells, and the construction often referred to as the Tjurina transform.
The title of this section is taken from the beautiful paper by Kleiman \cite{Kleiman69}.

Let $X$ be a scheme.
Let $E$ be a locally free $\OO_X$-module of finite rank $e$.
Let $n$ an integer such that $0\le n \le e$.
Let $G:= \Gr_n(E)$ be the Grassmannian of rank-$n$ quotients on $E$ over $X$.
Let
\begin{equation*}
\begin{tikzcd}
0 \ar[r] & K \ar[r,"\iota"] &
E_G \ar[r,"q"] & E_G/K \ar[r] & 0
\end{tikzcd}
\end{equation*}
be the tautological short exact sequence on $G$.

We begin by looking at the standard affine charts on $G$.

\begin{remark}
\label{rmk:chart-grass}
Suppose that $E = E'\oplus E''$, where $E'$ and $E''$ are free $\OO_X$-modules and $E'$ has rank $n$.
Let $U\subseteq G$ be the largest open subscheme where $q|_{E'} : E'_G\to E_G/K$ is an isomorphism.
Let $u$ denote the composition of the $\OO_{U}$-linear maps
\begin{equation*}
\begin{tikzcd}
E_U'' \ar[r,"q"] &
(E_G/K)_U \ar[r,"{(q|E')^{-1}}","\sim"'] &
E'_U.
\end{tikzcd}
\end{equation*}
We have an isomorphism of short exact sequences of $\OO_U$-modules
\begin{equation*}
\begin{tikzcd}
0 \ar[r] &
E''_U \ar[r,"\iota'"] \ar[d,"\sim"',dashed] &
E_U \ar[r,"q'"] \ar[d,equals]&
E'_U \ar[r] \ar[d,"q|E'", "\sim"'] &
0 \\
0 \ar[r] &
K_U \ar[r,"\iota"] &
E_U \ar[r,"q"] &
E_U/K_U \ar[r] &
0,
\end{tikzcd}
\end{equation*}
where 
\begin{equation*}
\iota ' =
\begin{bmatrix}
-u \\ \id
\end{bmatrix}
\qquad\text{and}\qquad
q' = 
\begin{bmatrix}
\id & u
\end{bmatrix}
\end{equation*}
relatively to the direct sum decomposition $E=E'\oplus E''$.
Fix bases on $E'$ and $E''$, so that $u$ may be identified with a matrix with $n$ rows, $e-n$ columns and entries in $\Gamma(U,\OO_U)$.
Let $\A_X^{n(e-n)}$ be the affine space of dimension $n(e-n)$ over $X$.
Let
\begin{equation*}
\psi : U\to \A_X^{n(e-n)}
\end{equation*}
be the unique morphism of $X$-schemes that pulls the coordinates on $\A_X^{n(e-n)}$ back to the $n(e-n)$ entries of the matrix representing $u$ (in some order).
It is a standard fact that $\psi$ is an isomorphism.
\end{remark}

We now turn to special Schubert cycles.
Let $A\subseteq E$ be locally free and locally split submodule of rank $a$, and let $p$ be a nonnegative integer.

\begin{definition}
\label{def:schubert}
The \emph{$p$th special Schubert cycle} associated to the subbundle $A\subseteq E$ is the closed subscheme $\sigma_p(A)\subseteq G$ defined as follows: if $p \le a$, then $\sigma_p(A)$ is the subscheme where the $\OO_G$-linear map
\begin{equation*}
\wedge^{a-p+1} (q|_A) : \wedge^{a-p+1} A_G\to
\wedge^{a-p+1} (E_G/K)
\end{equation*}
vanishes; otherwise, $\sigma_p(A)$ is empty.
\end{definition}

Intuitively, $\sigma_p(A)\subseteq G$ is the locus where $K\cap A_G$ has rank at least $p$.

\begin{proposition}
\label{prop:schubert-smooth}
The Schubert cell $\sigma_p(A)\setminus \sigma_{p+1}(A)$ is nonempty if, and only if, $0\le a-p\le n$. 
If these inequalities hold, then $\sigma_p(A)\setminus \sigma_{p+1}(A)$ is smooth of relative dimension $n(e-n) - (n-a+p)p$ over $X$.
\end{proposition}

Proposition \ref{prop:schubert-smooth} is, of course, standard.
We include here its reduction to the also standard Lemma \ref{lemma:prepared-chart} because later we will use both the lemma and an argument that the reduction isolates.

\begin{proof}
Let $w\in \sigma_p(A)\setminus \sigma_{p+1}(A)$ be a point, so that $q|_A : A_G\to E_G/K$ has rank $a-p$ at $w$.
Shrinking $X$ to a neighborhood of the image of $w$ in $X$, we may assume that $E$ is a free $\OO_X$-module and that $A$ is spanned by the first $a$ elements a basis of $E$.
Partitioning such a basis, we may find a direct sum decomposition 
\begin{equation*}
E = A'\oplus B' \oplus A'' \oplus B'',
\end{equation*}
where $A'$, $B'$, $A''$ and $B''$ are free $\OO_X$-modules such that $A'$ has rank $a-p$,
\begin{equation*}
A = A'\oplus A'',
\end{equation*}
and the natural $\OO_G$-linear map $(A'\oplus B')_G \to E_G/K$ is an isomorphism at $w$.
Applying Lemma \ref{lemma:prepared-chart} below, we obtain an open neighborhood $U$ of $w$ in $G$ such that 
$(\sigma_p(A)\setminus \sigma_{p+1}(A)) \cap U$ is isomorphic to the affine space 
\begin{equation*}
\A_X^{n(e-n)-p(n-a+p)}
\end{equation*}
over $X$.
\end{proof}

\begin{lemma}
\label{lemma:prepared-chart}
Suppose that
\begin{equation*}
E = A'\oplus B' \oplus A'' \oplus B'',
\end{equation*}
where $A'$, $B'$, $A''$ and $B''$ are free $\OO_X$-modules such that $A'$ has rank $a-p$ and $A = A'\oplus A''$.
Let $E' := A'\oplus B'$ and $E '' := A''\oplus B''$, so that $E = E'\oplus E''$.
Let $U \subseteq G$ and $u : E''_U\to E'_U$ be as in Remark \ref{rmk:chart-grass}. 
Let $u_{21}$ denote the composition of the $\OO_U$-linear maps
\begin{equation*}
\begin{tikzcd}
A''_U \ar[r,hook] &
(A''\oplus B'')_U \ar[r,"u"] &
(A' \oplus B')_U \ar[r, "\pr_2"] &
B'_U.
\end{tikzcd}
\end{equation*}
Then 
$(\sigma_p(A)\setminus \sigma_{p+1}(A)) \cap U$
is the subscheme of $U$ where $u_{21}=0$.
\end{lemma}

\begin{proof}
This can be proved with Lemma \ref{kleiman}.
\end{proof}

Finally, we consider the Tjurina transform.
Let $F$ be a locally free $\OO_X$-module of rank $f$.
Let $m:=\min(e,f)$ and let $i$ be a nonnegative integer.
Suppose that $n=m-i$, so that
\begin{equation*}
G := \Gr_n(E) = \Gr_{m-i}(E).
\end{equation*}

\begin{definition}
\label{def:tjurina} 
Let $\alpha : E\to F$ be an $\OO_X$-linear map.
The \emph{Tjurina transform} of the degeneracy locus $\Sigma^i(\alpha)\subseteq X$ is the subscheme $Z \subseteq G$ where $\alpha \circ \iota=0$.
\end{definition}

The strategy of proving results about a degeneracy locus by reducing them to assertions about its Tjurina transform is often called the ``Grassmannian trick''.
It is justified by the next proposition.

\begin{proposition}
\label{prop:tjurina}
There exists a unique morphism of schemes $\rho : Z\to \Sigma^i(\alpha)$ such that the following diagram commutes.
\begin{equation*}
\begin{tikzcd}
Z = \{\alpha\circ\iota =0 \}  \ar[hook]{r} \ar[dashed, "\rho"]{d} & G \ar[d,"\pi"]\\
\Sigma^i( \alpha) \ar[hook]{r} & X
\end{tikzcd}
\end{equation*}
This morphism is proper, surjective and induces an isomorphism of schemes
\begin{equation*}
Z\setminus \rho^{-1}(\Sigma^{i+1}(\alpha))
\xrightarrow\sim
\Sigma^i(\alpha)\setminus \Sigma^{i+1}(\alpha).
\end{equation*}
\end{proposition}

\begin{proof}
Uniqueness of $\rho$ is clear.
Existence follows from the fact that $\alpha_Z:E_Z\to F_Z$ factors though $(E_G/K)_Z$, which has rank $m-i$, so that $\wedge^{m-i+1} \alpha_Z=0$.

Now let $x : T\to \Sigma^i(\alpha)$ be a morphism of schemes. 
The set of morphisms of schemes $s : T\to Z$ such that $\rho\circ s = x$ is in bijection with the set $\mathscr L(x)$  of rank-$(e-n)$ locally free and locally split submodules $\tilde K\subseteq x^* E$ contained in $\ker(x^* \alpha)$.

If $T$ is the spectrum of a field, then $x^* \alpha : x^* E\to x^* F$ is a linear map of rank at most $m-i$, so $\ker(x^* E)\subseteq x^* E$ is a linear subspace of dimension at least $e-m+i = e-n$. 
Hence $\mathscr L(x)$ is nonempty in this case, which shows that $\rho$ is surjective.

If instead $T$ is arbitrary, but the image of $x$ is contained in $\Sigma^i(\alpha)\setminus \Sigma^{i+1}(\alpha)$, then $\ker(x^*\alpha) \subseteq x^* E$ is itself a locally free and locally split submodule of rank $(e-n)$ by Proposition \ref{coker-deg-loci}.
Hence $\mathscr L(x)$ is a singleton in this case, which shows that $\rho$ is an isomorphism away from $\Sigma^{i+1}(\alpha)$. 
\end{proof}

\section{Linear maps satisyfing symmetry and rank conditions}
\label{sec:sym-bil}

In this section and the next we will complete the proofs of the main results of this paper, Theorems \ref{thm:intro-smoothness} and \ref{thm:intro-2nd-sings}. 
Propositions \ref{prop:sym-deg-loci-2} and \ref{prop:sym-deg-loci-1} are all that remains to prove.
We will deduce both propositions from a single result, Theorem \ref{thm:univ-deg-bil}, which we will prove using the facts about Grassmannins reviewed in the preceding section.  

Let $X$ be a scheme.
Let $E$ and $F$ be finite, locally free $\OO_X$-modules of respective ranks $e$ and $f$.
Let $A\subseteq E$ be a rank-$a$ subbundle.
Let $\Box^2 E$ denote one of two $\OO_X$-modules: either $\Sym^2 E$ or $\wedge^2 E$.
Let $A\boxempty E$ be the image of $A\otimes E$ under the natural map $E\otimes E\twoheadrightarrow \Box^2 E$.

Thoughout this section, the symbol $\pm$ should be read as ``plus'' if $\Box^2 E= \Sym^2 E$, and as ``minus'' if $\Box^2 E = \wedge^2 E$.

\begin{lemma}
\label{lemma:box-free}
The $\OO_X$-module $A\boxempty E$ is a locally free of rank
\begin{equation*}
\tfrac 1 2 a(a\pm 1) + a (e-a).
\end{equation*}
\end{lemma}

\begin{proof}
The question being local on $X$, we may assume that $E$ is free with basis $\{v_1,\dotsc, v_e\} \subseteq \Gamma(X,E)$ and that $A$ is freely generated by $v_1,\dotsc,v_a$.
Then $A\boxempty E$ is freely generated by the images of the products $v_i \otimes v_j$ where $0\le i\le a$, $0\le j\le e$, and 
\begin{equation*}
\begin{cases}
i\le j &\text{if } \Box^2 E = \Sym^2 E \\
i < j & \text{if } \Box^2 E = \wedge^2 E.
\end{cases}
\end{equation*}
The result follows.
\end{proof}

Let $\pi : V\to X$ be the vector bundle corresponding to the locally free $\OO_X$-module $\sheafHom_X(A\boxempty E, F)$.
In symbols:
\begin{equation*}
V := \V(\sheafHom_X(A\boxempty E, F))
\end{equation*}
Let $\tilde h : (A\boxempty E)_V\to F_V$ be the tautological map.
Let
\begin{equation*}
h : E_V \to \sheafHom_X(A,F)_V
\end{equation*}
be the image of $h$ under the natural isomorphism
\begin{equation*}
\theta : \sheafHom_X(A\otimes E, F) \xrightarrow\sim
\sheafHom_X(E,\sheafHom(A,F)).
\end{equation*}

Fix nonnegative integers $i$ and $j$.
The object of the main result in this section, Theorem \ref{thm:univ-deg-bil}, is the locally closed subscheme $\Delta^{i,p} \subseteq V$ defined by
\begin{equation*}
\Delta^{i,p} := (\Sigma^i(h) \cap \Sigma^p(h|_A))\setminus(\Sigma^{i+1}(h)\cup \Sigma^{p+1}(h|_A)).
\end{equation*}
In this formula, $h|_A$ denotes the $\OO_V$-linear map $A_V \to \sheafHom_X(A, F)_V$ obtained by restricting $h$. 
Before stating the theorem, let us make a few preliminary observations.

\begin{remark}
\label{rmk:pts-of-sigma}
To understand the fibers of $\Delta^{i,p}\to X$ we may assume that $X$ is the spectrum of a field $k$. 
Then the set of $k$-rational points of $\Delta^{i,p}$ is in natural bijection with the set of $k$-linear maps
\begin{equation*}
\alpha : E \to \Hom_k(A, F)
\end{equation*}
such that 
\begin{enumerate}
\item the bilinear map $A\times A\to F$ that sends $(v, w)\mapsto h(v)(w)$ is symmetric if $\Box^2 E = \Sym^2 E$ and skew-symmetric otherwise;
\item $h$ has rank $\min(e,af)-i$; and
\item $h|_A$ has rank $a-p$. 
\end{enumerate}
\end{remark}

Let $n := \min(e, af)-i$.

\begin{lemma}
\label{lemma:nonempty}
The scheme $\Delta^{i,p}$ is nonempty if, and only if, 
\begin{enumerate}
\item $0 \le n$; 
\item $\max(a-n,0) \le p\le \min(a,e-n)$; and
\item $a-p$ is even if $\Box^2 E= \wedge^2 E$ and $f=1$.
\end{enumerate}
If $\Delta^{i,p}$ is nonempty, then the projection $\pi : \Delta^{i,p}\to X$ is surjective.
\end{lemma}

\begin{proof}
By considering the geometric fibers of the projection $\Delta^{i,p}\to X$, we may assume that $X$ is the spectrum of an algebraically closed field $k$. 

If $\Delta^{i,p}$ is nonempty, then (1), (2) and (3) hold by Remark \ref{rmk:pts-of-sigma} and the observation that a skew-symmetric matrix has necessarily even rank. 
Conversely, suppose that (1), (2) and (3) hold. By (3), one of the following alternatives holds:
\begin{itemize}
\item $\Box^2 E = \Sym^2 E$.
\item $\Box^2 E = \wedge^2 E$ and $a-p$ is even.
\item $\Box^2 E = \wedge^2 E$, $a-p$ is odd and $f\ge 2$.
\end{itemize}
Let us exhibit a point of $\Delta^{i,p}$ assuming the third alternative holds; the other two cases are slightly simpler and left to the reader.

Let $\{v_1,\dotsc, v_e \} \subseteq E$ be a basis of $E$ such that the vectors $v_1,\dotsc, v_a$ freely generate $A$.
Let $\{ v_1^\vee,\dotsc, v_e^\vee\} \subseteq E^\vee$ be the dual basis of $E^\vee$.
Let $w_1, w_2\in F$ be linearly independent vectors.
Then
\begin{equation*}
\alpha := \sum_{j=1}^{a-p-2}
(v_j^\vee\otimes v_{j+1}^\vee - v_{j+1}^\vee\otimes v_{j}^\vee)\otimes w_1
+ (v_{a-p}^\vee\otimes v_{1}^\vee - v_{1}^\vee\otimes v_{a-p}^\vee)
\otimes w_2
\end{equation*}
is an element of $\Hom_k(E, \Hom_k(A, F))$ that is contained in the image of $\Hom_k(A\boxempty E, F)$ and is such that $\alpha|_A$ has rank $a-p$. 
Let
\begin{equation*}
\beta_{1}, \dotsc, \beta_{n-a+p}
\in \Hom_k(A,F)
\end{equation*}
be maps which extend $\alpha(v_1), \dotsc, \alpha(v_{a-p})$ to a basis of an $n$-dimensional linear subspace of $\Hom_k(A, F)$. 
Then
\begin{equation*}
\alpha + \sum_{\ell=1}^{n-a+p} v_{a+\ell}^\vee \otimes \beta_\ell
\end{equation*}
is an element of $\Hom_k(E, \Hom_k(A, F))$ that corresponds to a closed point of $\Delta^{i,p}$.
\end{proof}

\begin{theorem}
\label{thm:univ-deg-bil}
If the scheme $\Delta^{i,p}$ is nonempty (see Lemma \ref{lemma:nonempty}), then it is smooth of pure relative codimension
\begin{equation*}
p(n-a+p) +f\cdot 
[ \tfrac 1 2 (-p^2 \pm p)
+ (e-n)a ] - n(e-n)
\end{equation*}
in $V$ over $X$.
\end{theorem}

\begin{proof}
Let $G := \Gr_n(E)$ be the Grassmannian of rank-$n$ quotients of $E$ over $X$.
Let
\begin{equation*}
\begin{tikzcd}
0 \ar[r] & K \ar[r,"\iota"] &
E_G \ar[r,"q"] & E_G/K \ar[r] & 0
\end{tikzcd}
\end{equation*}
be the tautological short exact sequence on $G$.
The fiber product $G' := G\times_X V$ is the Grassmannian of rank-$n$ quotients of $E_V$ over $V$. 
Let $Z\subseteq G'$ be the Tjurina transform of the degeneracy locus $\Sigma^i(h)\subseteq V$.
Thus $Z$ is the closed subscheme of $G'$ where $h \circ \iota = 0$.
Let $W$ denote the scheme-theoretic intersection of $Z$ with the Schubert cell
\begin{equation*}
(\sigma_p(A_V)\setminus \sigma_{p+1}(A_V)) =
(\sigma_p(A)\setminus \sigma_{p+1}(A))\times_X V
\subseteq G'.
\end{equation*}

By Proposition \ref{prop:tjurina}, the second projection $\pr_2 : G'\to V$ induces an isomorphism
\begin{equation*}
Z\setminus \pr_2^{-1} (\Sigma^{i+1}(h) ) \xrightarrow\sim \Sigma^i(h)\setminus \Sigma^{i+1}(h).
\end{equation*}
We claim this isomorphism maps $W\setminus \pr_2^{-1}(\Sigma^{i+1}(h))$ onto $\Delta^{i,p}$.
To see this, let $\tilde Z : = Z\setminus \pr_2^{-1} (\Sigma^{i+1}(h) )$. By the proof of Proposition \ref{prop:tjurina} the natural inclusion
$
K_Z \hookrightarrow \ker(h_Z)
$
is an isomorphism over $\tilde Z$.
This implies that $h_{\tilde Z}$
factors through a locally split injection
\begin{equation*}
(E_G/K)_{\tilde Z} \hookrightarrow \sheafHom_X(A,F)_{\tilde Z},
\end{equation*}
which in turn implies that the subschemes of $\tilde Z$ where $\wedge^{a-p'+1}(q|_A)=0$ and where $\wedge^{a-p'+1}(h|_A) =0$ coincide for all $0\le p'\le a+1$.
The claim follows.

Thus it suffices to show that $W$ is either empty or smooth of pure relative codimension
\begin{equation*}
p(n-a+p) +f\cdot 
[ \tfrac 1 2 (-p^2 \pm p)
+ (e-n)a ]
\end{equation*}
in $G'=G\times_S V$ over $X$.

Suppose that $W$ is nonempty and let $w\in W$ be a point.
Then $\pr_1(w)\in \sigma_p(A)\setminus \sigma_{p+1}(A) \subseteq G$.
As in the proof of Proposition \ref{prop:schubert-smooth}, after shrinking $X$ to a neighborhood of the image of $w$ in $X$, we may assume that the $\OO_X$-module $F$ is free, and find a direct sum decomposition
\begin{equation*}
E = A' \oplus B' \oplus A'' \oplus B'',
\end{equation*}
where $A'$, $B'$, $A''$ and $B''$ are free $\OO_X$-modules such that $A'$ has rank $a-p$,
\begin{equation*}
A=A'\oplus A'',
\end{equation*}
and the natural $\OO_G$-linear map $(A'\oplus B')_G \to E_G/K$ is an isomorphism at $\pr_1(w)$. 
Let $E' := A'\oplus B'$ and $E'':= A''\oplus B''$, so that $E=E'\oplus E''$. 
Let
\begin{equation*}
U\subseteq G, \quad u: E_U''\to E'_U
\quad\text{and}\quad
\iota' : E_U''\to E_U
\end{equation*}
be as in Remark \ref{rmk:chart-grass}.

Let
\begin{equation*}
u \sim 
\begin{bmatrix}
u_{11} & u_{12} \\
u_{21} & u_{22}
\end{bmatrix}
\end{equation*}
be the block decomposition of the $\OO_U$-linear map $u : E''_U\to E'_U$ coming from the direct sum decompositions of its source and target.
Thus
\begin{align*}
u_{11} &: A''_U \to A'_U &
u_{12} &: B''_U \to A'_U \\
u_{21} &: A''_U \to B'_U &
u_{22} &: B''_U \to B'_U 
\end{align*}
are $\OO_U$-linear maps and, for example, $u_{21}$ is the restriction of $u : E''_U\to E'_U$ to $A''_U$ followed by the second projection $E'_U= A'_U\oplus B'_U \to B'_U$.

The fiber product $U\times_X V$ is an open subset of $G'$ that contains $w$.
By Lemma \ref{lemma:prepared-chart}, the intersection $W\cap (U\times_X V)$ is the closed subscheme of $U\times_X V$ where
\begin{equation*}
u_{21} = 0
\qquad\text{and}\qquad
h \circ \iota' =0.
\end{equation*}

Fix a basis for the free $\OO_X$-module $F$ and let $\mathcal B \subseteq \Gamma(X,F^\vee)$ be the corresponding dual basis.
The subscheme of $U\times_X V$ where $h\circ \iota'=0$ is the same as the subscheme where $\tau h\circ \iota'=0$ for all $\tau\in \mathcal B$.

Given a basis element $\tau\in \mathcal B$, let 
\begin{equation*}
\tau h \sim
\begin{bmatrix}
\alpha_{11}^\tau & \beta_{11}^\tau & \alpha_{12}^\tau & \beta_{12}^\tau \\
\alpha_{21}^\tau & \beta_{21}^\tau & \alpha_{22}^\tau & \beta_{22}^\tau
\end{bmatrix}
\end{equation*}
be the block decomposition of the $\OO_V$-linear map $\tau h : E_V \to A^\vee_V$ coming from the direct sum decompositions $E = A' \oplus B' \oplus A'' \oplus B''$ and $A = A' \oplus A''$. 
Thus
\begin{align*}
\alpha_{11}^\tau &: A'_V\to (A'_V)^\vee &
\beta_{11}^\tau &: B'_V\to (A'_V)^\vee \\
\alpha_{12}^\tau &: A''_V\to (A'_V)^\vee &
\beta_{12}^\tau &: B''_V\to (A'_V)^\vee \\
\alpha_{21}^\tau &: A'_V\to (A''_V)^\vee &
\beta_{21}^\tau &: B'_V\to (A''_V)^\vee\\
\alpha_{22}^\tau &: A''_V\to (A''_V)^\vee &
\beta_{22}^\tau &: B''_V\to (A''_V)^\vee
\end{align*}
are $\OO_V$-linear maps.
The maps $\alpha_{rs}^\tau$ are either symmetric or skew-symmetric, because $h =\theta(\tilde h)$ and $\tilde h \in \Hom_V((A\boxempty E)_V, F_V)$. In particular:
\begin{align*}
(\alpha_{11}^\tau)^\vee &= \pm\alpha_{11}^\tau &
(\alpha_{22}^\tau)^\vee &= \pm\alpha_{22}^\tau &
(\alpha_{12}^\tau)^\vee &= \pm\alpha_{21}^\tau
\end{align*}

Having these block decompositions in place, we can say that $W$ is the subscheme of $U\times_X V$ where $u_{21}=0$ and 
\begin{equation*}
\begin{bmatrix}
\alpha_{11}^\tau & \beta_{11}^\tau & \alpha_{12}^\tau & \beta_{12}^\tau \\
\alpha_{21}^\tau & \beta_{21}^\tau & \alpha_{22}^\tau & \beta_{22}^\tau
\end{bmatrix}
\begin{bmatrix}
-u_{11} & -u_{12} \\
-u_{21} & -u_{22} \\
1 & 0 \\
0 & 1
\end{bmatrix} =
\begin{bmatrix}
0 & 0 \\
0 & 0
\end{bmatrix}
\end{equation*}
for all $\tau\in \mathcal B$.
Working out the matrix product and using the equation $u_{21} = 0$ and the symmetries from the preceding paragraph, we obtain the following system of equations, which also define $W$ as a subscheme of $U\times_X V$:
\begin{align}
u_{21} &= 0 \label{eqn:deg-bil-1}\\
\alpha_{12}^\tau &= \alpha_{11}^\tau u_{11} \label{eqn:deg-bil-2}\\
\alpha_{22}^\tau &= u_{11}^\vee \alpha_{11}^\tau u_{11} \label{eqn:deg-bil-3}\\
\beta_{12}^\tau &= \alpha_{11}^\tau u_{12} + \beta_{11}^\tau u_{22} \label{eqn:deg-bil-4}\\
\beta_{22}^\tau &= u_{11}^\vee \alpha_{11}^\tau u_{12} + \beta_{21}^\tau u_{22} \label{eqn:deg-bil-5}
\end{align}
Equations (\ref{eqn:deg-bil-2}-\ref{eqn:deg-bil-5}) depend on a basis element $\tau \in \mathcal B$ and must hold for all such elements. 

Fix $\OO_X$-linear bases on $A', B', A''$, and $B''$.
Then, by Remark \ref{rmk:chart-grass} and Example \ref{vb-trivial}, we may identify the schemes $U$ and $V$ with affine spaces over $X$;
the $\OO_U$-linear map $u : E''_U\to E'_U$ with a matrix whose entries are the coordinates on $U$;
and the $\OO_V$-linear map $\tilde h : (E\boxempty A)_V \to F_V$ with a matrix whose entries are the coordinates on $V$.

With these identifications, (\ref{eqn:deg-bil-2}-\ref{eqn:deg-bil-5}) become equations between matrices with entries in $\Gamma(U\times_X V, \OO_{U\times V})$.
Let us make four observations about these matrices.
First, the entries of the matrix on the left-hand side of (\ref{eqn:deg-bil-1}) are distinct coordinates on $U$.
Second, the entries of the matrices on the right-hand sides of (\ref{eqn:deg-bil-2}-\ref{eqn:deg-bil-5}) are distinct coordinates on $V$, as much as that is allowed by the symmetry or skew-symmetry of the matrices $\alpha_{22}^\tau$.
Third, the set of coordinates appearing in the left-hand sides of the equations (\ref{eqn:deg-bil-1}-\ref{eqn:deg-bil-5}) is disjoint from the set of coordinates appearing the right-hand sides of these equations.
And fourth, for each $\tau\in \mathcal B$, the right-hand side of (\ref{eqn:deg-bil-3}) has the same type of symmetry as $\alpha_{22}^\tau$.

These observations imply that $W\cap (U\times_X V)$ is isomorphic to an affine space over $X$.
Indeed, an isomorphism is given by the coordinates on $U\times_X V$ that do not appear in the left-hand sides of (\ref{eqn:deg-bil-1}-\ref{eqn:deg-bil-5}).
It follows that $W$ is smooth over $X$, because $W\cap (U\times_X V)$ was constructed as a neighborhood of the arbitrary point $w\in W$.

Let us compute the relative codimension of $W$ in $G' := G\times_X V$.
The equations (\ref{eqn:deg-bil-1}-\ref{eqn:deg-bil-5}) relate maps betwen the pullbacks to $V$ of $A'$, $A''$, $B'$, $B''$ and the duals of these $\OO_X$-modules.
The respective ranks of $A'$, $A''$, $B'$ and $B''$ are $a-p$, $p$, $n-a+p$ and $e-n+p$. 
\begin{itemize}
\item Equation (\ref{eqn:deg-bil-1}) is between elements of $\Hom_V(A''_V, B'_V)$ and contributes $p(n-a+p)$ to the relative codimension of $W$ in $G'$.
\item Equation (\ref{eqn:deg-bil-2}) is between elements of $\Hom_V(A''_V, (A'_V)^\vee)$ and contributes $p(a-p)$ to the relative codimension of $W$ in $G'$ for each $\tau \in \mathcal B$.
\item Equation (\ref{eqn:deg-bil-3}) is between symmetric or skew-symmetric elements of $\Hom_V(A''_V, (A''_V)^\vee)$ and contributes $\tfrac 1 2 p(p\pm 1)$ to the relative codimension of $W$ in $G'$ for each $\tau \in \mathcal B$.
\item Equation (\ref{eqn:deg-bil-4}) is between elements of $\Hom_V(B''_V, (A'_V)^\vee)$ and contributes $(e-n-p)(a-p)$ to the relative codimension of $W$ in $G'$ for each $\tau \in \mathcal B$.
\item Equation (\ref{eqn:deg-bil-5}) is between elements of $\Hom_V(B''_V, (A''_V)^\vee)$ and contributes $(e-n-p)p$ to the relative codimension of $W$ in $G'$ for each $\tau \in \mathcal B$.
\end{itemize}
We conclude that $W$ has relative codimension
\begin{equation*}
p(n-a+p) +f\cdot 
[ p(a-p)
+ \tfrac 1 2 p(p\pm 1)
+ (e-n-p)a ]
\end{equation*}
in $G'$ over $X$, as we set out to show.
\end{proof}

\begin{proposition}[{\cite[Proposition 2.53]{braune-thesis}}]
\label{prop:sym-deg-loci-1}
Suppose that $A=E$.
Then the constant-rank locus 
\begin{equation*}
\Sigma^i(h)\setminus \Sigma^{i+1}(h)\subseteq 
V=\V(\sheafHom_X(\Box^2 E,F))
\end{equation*}
is nonempty if, and only if,
\begin{enumerate}
\item $i\le e$; and
\item $e-i$ is even if $\Box^2 E = \wedge^2 E$ and $f=1$.
\end{enumerate}
In this case, $\Sigma^i(h)\setminus \Sigma^{i+1}(h)$ is smooth of pure relative codimension
\begin{equation*}
i(e-i)(f-1) + \tfrac 1 2 i(i\pm 1) f
\end{equation*}
in $V$ over $X$.
\end{proposition}

\begin{proof}
Follows from Lemma \ref{lemma:nonempty} and Theorem \ref{thm:univ-deg-bil}, because $\Sigma^i(h)=\Delta^{i,i}$ when $A=E$. 
\end{proof}
  
\begin{proposition}
\label{prop:sym-deg-loci-2}
Suppose that $af\le e$.
If $\Box^2 E = \Sym^2 E$, then the first degeneracy locus 
\begin{equation*}
\Sigma^1(h) \subseteq V = \V(\sheafHom_X(A\boxempty E, F))
\end{equation*}
has relative codimension $e-af+1$ in $V$ over $X$.
If $\Box^2 E = \wedge^2 E$, then the same holds provided that 
\begin{enumerate}
\item $a>1$; and
\item if $f=1$, then $a=e$ and $e$ is even.
\end{enumerate}
Otherwise, $\Sigma^1(h)$ has relative codimension $e-af$ in $V$ over $X$.
\end{proposition}

\begin{proof}
The support of the degeneracy locus $\Sigma^1(h)\subseteq V$ is the disjoint union of the supports of the subschemes $\Delta^{i',p'} \subseteq V$ with $i'\ge 1$ and $p'\ge 0$. In symbols:
\begin{equation*}
|\Sigma^1(h)| = \bigsqcup_{i'\ge 1;~ p'\ge 0} |\Delta^{i',p'}|
\end{equation*}
By Lemma \ref{lemma:nonempty}, Theorem \ref{thm:univ-deg-bil} and Lemma \ref{lemma:minimal} below, the minimum of the codimensions of the subschemes $\Delta^{i,p}\subseteq V$ with $i\ge 1$ and $p\ge 0$ is either $e-af+1$ or $e-af$, as in the statement of the proposition.
\end{proof}

\section{The minimal codimension}
\label{sec:min-codim}

The following lemma was used in the proof of Proposition \ref{prop:sym-deg-loci-2}.
Its proof consists of tedious, but straightforward, case-by-case analysis. 

\begin{lemma}
  \label{lemma:minimal}
Let $e, a, f$ be positive integers such that $af\le e$.
Let $R$ be the polygonal region
\begin{equation*}
\{ (i,p)\in \R^2 ~:~ 1\le i\le af;~
\max(a-af+i,0) \le p\le \min(a,e-af+i) \},
\end{equation*}
see Figure \ref{fig:R}.
Let $C_+ : \R^2 \to \R$ and $C_- : \R^2 \to \R$ be the  functions defined by
\begin{equation*}
C_\pm (i,p) = p(af-i-a+p)
+ f\cdot [ \tfrac 1 2 (-p^2 \pm p) 
+ (e-af+i)a ]
- (af-i)(e-af+i).
\end{equation*}
\begin{enumerate}
\item The minimum value achieved by $C_+$ on $R\cap \Z^2$ is $e-af+1$.

\item If $f>1$, then the minimum value achieved by $C_-$ on $R\cap \Z$ is
\begin{equation*}
\begin{cases}
e-af+1  & \text{if $a>1$} \\
e-af    & \text{if $a=1$.}
\end{cases}
\end{equation*}

\item If $f=1$, then the minimum value achieved by $C_-$ on $R\cap (\Z\times(a + 2\Z))$ is
\begin{equation*}
\begin{cases}
e-af+1  & \text{if $a=e$ and $e$ is even} \\
e-af    & \text{otherwise.}
\end{cases}
\end{equation*}
\end{enumerate}
\end{lemma}

\begin{figure}
\includegraphics{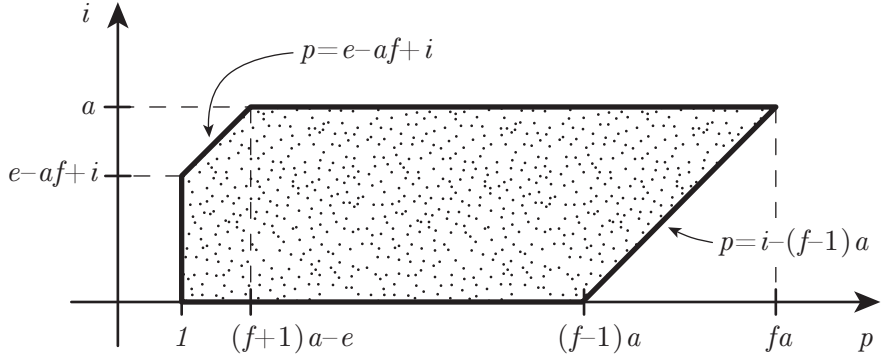}
\caption{The polygonal region of Lemma \ref{lemma:minimal}.}
\label{fig:R}
\end{figure}

\begin{proof}
If $(i,p)\in R$, then $1\le i$ and $p\le e-af+i$, so
\begin{align*}
\dfrac{\partial C_\pm}{\partial i}(i,p) = e-af+2i-p > 0.
\end{align*}
This implies that the various minima are achieved on the union of two possibly degenerate line segments $L_1\cup L_2$ in the boundary of $R$.
Namely:
\begin{align*}
L_1 &= \{(1,p) ~|~ \max(a-af+1, 0) \le p \le \min(a, e-af+1)\} \\
L_2 &= \{ (i, e-af+i)~|~ 1\le i \le (f+1)a-e \}
\end{align*}
Let $m := \max(a-af+1, 0)$ and $M := \min(a, e-af+1)$.

To study the restrictions of $C_\pm$ to $L_1$ and $L_2$, let $q_{\pm, 1} : \R\to \R$ and $q_{\pm, 2} : \R\to\R$ be the functions defined by
\begin{equation*}
q_{\pm, 1}(p) = C_\pm(1, p)
\qquad\text{and}\qquad
q_{\pm, 2} = C_\pm(i, e-af+i).
\end{equation*}
Let $I_1 := [m, M]$ and $I_2 := [1, (f+1)a -e]$. 
For $j=1, 2$, we wish to minimize $q_{+, j}$ over $I_j\cap \Z$, and $q_{-,j}$ over
\begin{equation*}
\begin{cases}
I_j\cap \Z & \text{if $f>1$}\\
I_j\cap (a+2\Z) & \text{if $f=1$}.
\end{cases}
\end{equation*}
The quadratic functions $q_{\pm, j}$ are convex if $f=1$, linear or constant if $f=2$, and concave if $f\ge 2$.
We consider each case separately.

Suppose that $f=1$.
To minimize a convex quadratic function over a finite set, we look for the point in the set that is closest to the global minimizer of the function.
The global minimizer of $q_{\pm, 1}$ is $p=1\mp \tfrac 1 2$, while that of $q_{\pm, 2}$ is $i = -(e-a)\mp \tfrac 1 2$.
Therefore:
\begin{itemize}
\item The restriction of $q_{+, 1}$ to $I_1\cap \Z = [1, M]\cap \Z$ achieves its minimum at $p=1$.
\item The restriction of $q_{-, 1}$ to $I_1 \cap (a + 2\Z)$ achieves its minimum at $p=1$ if $a$ is odd, and at $p=2$ if $a$ is even.
\item The restriction of $q_{+, 2}$ to $I_2\cap \Z = [1, 2a-e]\cap \Z$ achieves its maximum at $i=1$.
\item The restriction of $q_{-, 2}$ to $I_2\cap (a + 2\Z)$ achieves its minimum at $i=2$ if $a=e$ and $e$ is even, and at $i=1$ otherwise.
\end{itemize}

Next, suppose that $f=2$.
\begin{itemize}
\item If $a=1$, then $q_{\pm, 1}$ is strictly decreasing and $I_1 = [0, 1]$, so the restriction of $q_{\pm, 1}$ to $I_1 \cap \Z$ achieves its minimum at $p=1$.
\item If $a>1$, then $q_{\pm, 1}$ is nondecreasing and $I_1 = [0, M]$, so the restriction of $q_{\pm, 1}$ to $I_1\cap \Z$ achieves its minimum at $p=0$. 
\item The function $q_{\pm, 2}$ is nondecreasing, so the restriction of $q_{\pm, 2}$ to $I_2\cap \Z$ achieves its minimum at $i=1$.
\end{itemize}

Finally, suppose that $f\ge 3$. 
To minimize a concave quadratic function over a finite set, we look for the point in the set that is furthest from the global maximizer of the function.
\begin{itemize}
\item If $a=1$, then the global maximizer of of $q_{\pm, 1}$ lies to the left of the midpoint of the interval $I_2=[0,1]$, so the restriction of $q_{\pm, 1}$ to $I_2\cap \Z$ achieves its minimum at $p=1$.
\item If $a>1$, then the global maximizer of of $q_{\pm, 1}$ lies to the right of the midpoint of the interval $I_2=[0,M]$, so the restriction of $q_{\pm, 1}$ to $I_2\cap \Z$ achieves its minimum at $p=0$.
\item The global maximizer of $q_{\pm, 2}$ lies to the right of the midpoint of the interval $I_2 = [1, (f+1)a-e]$, so the restriction of $q_{\pm, 2}$ to $I_2\cap \Z$ achieves its minimum at $i=1$. \qedhere
\end{itemize}
\end{proof}

\section{Power series with finite Milnor number}

In this section we state basic facts about power series with finite Milnor number and use these facts to prove a version of Morse's Lemma with Parameters, namely Proposition \ref{prop:morse-params}. 
In the next section we will use this proposition to prove Theorem \ref{thm:intro-morse} from the introduction.

Let $k$ be a field. 
Let $x=(x_1,\dotsc,x_n)$ be a finite set of indeterminates.
Let $f\in k[[x]]$ be a power series.

\begin{definition}
The \emph{Jacobian ideal} of $f$, denoted $\operatorname{jac}(f)$, is the ideal generated in the power series ring $k[[x]]$ by the partial derivatives $\partial f/\partial x_1,\dotsc, \partial f/\partial x_n$.
The quotient $k[[x]]/\operatorname{jac}(f)$ is called the \emph{Milnor algebra} of $f$.
Its (possibly infinite) dimension as a vector space over $k$ is called the \emph{Milnor number} of $f$ and denoted by $\mu(f)$.
\end{definition}

\begin{definition}
Let $r$ be a positive integer.
We say that a power series $f\in k[[x]]$ is \emph{$r$-determined} if for every power series $g\in k[[x]]$ such that $f-g\in \langle x\rangle^{r+1}$, there exists an automorphism of $k[[x]]$ as a local $k$-algebra that sends $g$ to $f$.
We say that $f$ is \emph{finitely determined} if it is $r$-determined for some $r\ge 1$.
\end{definition}

\begin{proposition}
\label{DeterminacyBound}
If $f\in k[[x]]$ has finite Milnor number, then $f$ is finitely determined.
More precisely, let $r$ be the largest positive integer such that $\langle x\rangle^r\subseteq \operatorname{jac}(f)$.
Then $f$ is $2r$-determined.
\end{proposition}

\begin{proof}
This result follows from \cite[Theorem 2.1]{BGM12}.
For a simple, direct argument, see the proof of \cite[Lemma 10.8]{Milnor68}.
\end{proof}

The analogue of Proposition \ref{DeterminacyBound} for germs of smooth functions on Euclidean space is a very special case of \cite[Theorem 1.2]{Wall81}.

Let $\CC$ be the category whose objects are complete, Noetherian, local $k$-algebras with residue field $k$, and whose morphisms are maps of local $k$-algebras.

\begin{definition}
\label{def:unfolding}
Let $R$ be a complete local $k$-algebra in $\CC$. 
\begin{enumerate}
\item An \emph{unfolding} (or \emph{deformation}) of $f$ over $R$ is a power series $F\in R[[x]]$ that maps to $f\in k[[x]]$ under the quotient map $R\to k$.
\item Let $F, F' \in R[[x]]$ be unfoldings of $f$ over $R$.
A \emph{right-equivalence} (or \emph{morphism}) $F\to F'$ is a local $R$-algebra map $\varphi : R[[x]]\to R[[x]]$ that lifts the identity of $k[[x]]$ and sends $F$ to $F'$.
\end{enumerate}
\end{definition}

Unfoldings of $f$ over $R$ and right-equivalences between them form a category (in fact, a groupoid) that we denote by $\mathscr D(R)$.
A map $b: R\to R'$ of complete local $k$-algebras in $\CC$ induces an obvious functor functor $b_*: \mathscr D(R)\to \mathscr D(R')$.

\begin{definition}
The \emph{functor of unfoldings} of $f$ is the functor
\begin{equation*}
D: \CC\to (\mathrm{Sets})
\end{equation*}
that sends a complete local $k$-algebra $R\in \CC$ to the set $D(R)$ of right-equivalence classes of unfoldings of $f$ over $R$, and acts on morphisms in the obvious way.
\end{definition}

\begin{definition}
Let $R$ be a complete local $k$-algebra in $\CC$. Let $F\in R[[x]]$ be a unfolding of $f$ over $R$. We say that $F$ is \emph{right-complete} (or \emph{versal}) if, for every complete local $k$-algebra $A$ in $\CC$, the map 
\begin{equation*}
\Hom_{\CC}(R,A)\to D(A)
\end{equation*}
that sends $b\mapsto b_*F$ is surjective.
\end{definition}

\begin{proposition}
\label{versal-deform}
Suppose that $f$ has finite Milnor number.
Let $g_1, \dotsc, g_\mu \in k[[x]]$ be power series whose images span the Milnor algebra $k[[x]]/\jac(f)$ as a vector space over $k$.
Let $s= (s_1,\dotsc, s_\mu)$ be a set of $\mu$ indeterminates.
Then
\begin{equation*}
F := f + s_1 g_1 + \dotsb + s_\mu g_\mu \in k[[s,x]]
\end{equation*}
is a right-complete unfolding of $f$ over $k[[s]]$.
\end{proposition}

\begin{proof}
Let $A\in \CC$ be a local $k$-algebra.
Let $P\twoheadrightarrow A$ be a surjective map of local $k$-algebras, where $P$ is a ring of power series in finintely many variables and coefficients in $k$.
We have a commutative diagram
\begin{equation*}
\begin{tikzcd}
\Hom_\CC(k[[s]], P) \ar[r] \ar[d, two heads] &
D(P) \ar[d, two heads] \\
\Hom_\CC(k[[s]], A) \ar[r] &
D(A)
\end{tikzcd}
\end{equation*}
where the vertical maps are surjective and the horizontal maps are induced by $F$.
It suffices to show that the top horizontal map is surjective.
This can be done using the method of the proof of \cite[Corollary 1.17]{GLS2007}.
\end{proof}

The analogue of Proposition \ref{DeterminacyBound} for unfoldings of germs of smooth functions on Euclidean space is a very special case of \cite[Theorem 3.4]{Wall81}.

We now turn to generalization of Morse's lemma that we will use in the next section.
Suppose that $f\in \langle x\rangle^2$ and that the Hessian matrix of $f$ has rank $r$ at the origin. Then $r$ is even if $p=2$.
Let
\begin{equation*}
q=\begin{cases}
x_1^2 + \dotsb + x_r^2 & \text{if $p\ne 2$}\\
x_1x_2 + \dotsb + x_{r-1}x_r & \text{if $p=2$}.
\end{cases}
\end{equation*}

\begin{lemma}
\label{morse-approximation}
If $p\ne 2$, then there exists a local $k$-algebra automorphism $\varphi:k[[x]]\to k[[x]]$ such that $\varphi(f)\equiv q$ modulo $\langle x\rangle^3$.
If $p= 2$, then there exists a local $k$-algebra  automorphism $\varphi:k[[x]]\to k[[x]]$ such that either $\varphi(f)\equiv q$ or $\varphi(f)\equiv q+x_{r+1}^2$ modulo $\langle x\rangle^3$.
\end{lemma}

\begin{proof}
Let $\mathfrak n$ denote the maximal ideal $\langle x\rangle \subset k[[x]]$. Let $q(f)$ denote the image of $f\in \mathfrak n^2$ inside $\mathfrak n^2/\mathfrak n^3 = \Sym^2 (\mathfrak n/\mathfrak n^2)$. Then $q(f)$ is a quadratic form whose associated bilinear form is represented by the Hessian matrix of $f$ at the origin. 
By the classification of quadratic forms, there exists a $k$-linear automorphism $\varphi_1$ of $\mathfrak n/\mathfrak n^2$ such that $\Sym^2(\varphi_1)$ sends $q(f)$ to either $q$ or $q+x_{r+1}^2$.
We may take $\varphi$ to be the local $k$-algebra automorphism of $k[[x]]$ induced by $\varphi_1$, which characterized by the following property: for all nonnegative integers $i$, the self-map of $\mathfrak n^i/\mathfrak n^{i+1} = \Sym^i(\mathfrak n/\mathfrak n^2)$ induced by $\varphi$ is equal to $\Sym^i(\varphi_1)$.
\end{proof}

\begin{lemma}[Morse's Lemma]
\label{lemma:morse}
If $r=n$, then there exists an automorphism of $k[[x]]$ as a local $k$-algebra that maps $f$ to $q$.
\end{lemma}

\begin{proof}
Because $r=n$, we have $\langle x\rangle = \operatorname{jac}(q)$. It follows from Proposition \ref{DeterminacyBound} that $q$ is 2-determined. Hence it suffices to show that there exists an automorphism of $k[[x]]$ as local $k$-algebra that sends $f$ to $q$ modulo $\langle x\rangle^3$. This follows from Lemma \ref{morse-approximation} above.
\end{proof}

\begin{proposition}[Morse's Lemma with Parameters]
\label{prop:morse-params}
Let $R$ be a complete local $k$-algebra with residue field $k$.
Let $F\in R[[x]]$ be a power series with residue $f$ in $k[[x]]$.
Then there exist a power series $h\in R[[x_{r+1},\dotsc,x_n]]$ and an automorphism of $R[[x]]$ as a local $R$-algebra that sends $F$ to $q+h$.
\end{proposition}

\begin{proof}
By Lemma \ref{morse-approximation}, there exists a local $k$-algebra automorphism of $k[[x]]$ that maps $f$ to either $q$ or $q + x_{r+1}^2$ modulo $\langle x\rangle^3$.
Lifting such an automorphism to a local $R$-algebra automorphism of $R[[x]]$ and replacing $F$ with its image under the lift, we may assume that $f$ is congruent to either $q$ or $q + x_{r+1}^2$ modulo $\langle x\rangle^3$.

Let $R'$ denote the complete local $k$-algebra $R[[x_{r+1},\dotsc,x_n]]$.
Let $\bar f$ denote the image of $f$ under the map $k[[x_1,\dotsc,x_n]]\to k[[x_1,\dotsc,x_r]]$ that sends $x_i\mapsto x_i$ for $i\le r$ and $x_i\mapsto 0$ for $i>r$.
After replacing $R$ by $R'$ and $f$ by $\bar f$, we may assume that $r=n$ and $f\equiv q$ modulo $\langle x\rangle^3$.

By Morse's Lemma (Lemma \ref{lemma:morse}), there exists a local $k$-algebra automorphism of $k[[x_1,\dotsc,x_r]]$ that sends $f$ to $q$. After lifting such an automorphism to a local $R$-algebra automorphism of $R[[x]]$, we may assume that $f=q$. In other words, we may assume that $F$ is a unfolding of $q$ over $R$.

By Proposition \ref{versal-deform} and the assumption that $r=n$, the power series
\begin{equation*}
q + t\in k[[t,x]],
\end{equation*}
is a versal unfolding of $q$ over $k[[t]]$.
We may therefore find a map of local $k$-algebras $a : k[[t]]\to R$ and a right-equivalence of unfoldings $\varphi : q+ a(t) \to F$. The element $h := a(t)\in R$ and the automorphism of $R[[x]]$ underlying $\varphi$ satisfy the conclusions of the proposition.
\end{proof}

\section{Local description of corank-1 singularities}

\label{sec:proof-intro-morse}

\begin{proof}[Proof of Theorem \ref{thm:intro-morse}]
We note that $r := \dim_{f(x)}Y$ and $n := \dim_x X$.
The assumptions that $x\in \Sigma^1(f)$ and $r \le n$ imply that the differential $df(x) : \T_X(x)\to \T_Y(y)$ has rank $r-1$.
Using this, we can find a reordering of the coordinates $y_1,\dotsc, y_r\in \OO_{Y,y}$ and a system of parameters $x_1,\dotsc,x_n\in \OO_{X,x}$ with the desired properties.

The integer $n-r+1-j$ is the rank of the second intrinsic differential
\begin{equation*}
\dd_x^2 f : \ker(df(x)) \to \Hom_k(\ker(df(x)),\coker(df(x))).
\end{equation*}
The kernel of the differential $df(x) : \T_X(x)\to \T_Y(x)$ is freely generated by the vectors $\partial/\partial x_a$ with $a=r,\dotsc, n$, while its cokernel is freely generated by the image of $\partial/\partial y_r$. 
By Remark \ref{rmk:2nd-id-locally} we have
\begin{equation*}
\dd_x^2 f\left(\dfrac \partial {\partial x_a}\right) 
= \sum_{b=r}^n \dfrac{\partial^2 f_r}{\partial x_a \partial x_b} (x) \cdot dx_b \otimes \dfrac \partial {\partial y_r}
\end{equation*}
for all $a=r,\dotsc, n$.
Thus $\dd_x^2 f$ is represented by the square submatrix of size $n-r+1$ in the bottom-right corner of the Hessian matrix of $f_r$.
If $k$ has characteristic 2, then this submatrix is skew-symmetric, and therefore its rank $n-r+1-j$ is even.

Write $f_r = f_r(x) + g_1 + g_2$, where $g_1$ is homogeneous polynomial of degree 1 in $x_1,\dotsc,x_n$, and $g_2\in \langle x_1,\dotsc, x_n\rangle^2$.
Then $g_1$ only involves the variables $x_1,\dotsc,x_{r-1}$ by the assumption the differential $df(x) : \T_X(x)\to \T_Y(x)$ has rank $r-1$.
Let $\bar g_2 := g_2(0,\dotsc,0,x_r,\dotsc,x_n)\in k[[x_r,\dotsc,x_n]]$.
The Hessian matrix of $\bar g_2$ is the square submatrix of size $n-r+1$ in the bottom-right corner of the Hessian matrix of $f_r$.
By the preceding paragraph, it has rank $n-r+1-j$.
Viewing $g_2$ as an unfolding of $\bar g_2$ over $R:= k[[x_1,\dotsc,x_{r-1}]]$ and applying Morse's Lemma with Parameters (Proposition \ref{prop:morse-params}), we may find an automorphism $\varphi$ of $k[[x_1,\dotsc,x_n]]$ as a local $k[[x_1,\dotsc,x_{r-1}]]$-algebra that sends $g_2$ to $q + h'$ for some power series $h'\in k[[x_1,\dotsc,x_n]]$ that does not involve the variables $x_r,\dotsc,x_{n-j}$. Setting $h := g_1 + h'$, the result follows. 
\end{proof}

% \bibliographystyle{amsplain}
% \bibliography{../MasterBib}

\end{document}